\documentclass[12pt, reqno, a4paper,oneside]{amsart}
\usepackage{graphicx, epsfig, psfrag}
\usepackage{amsfonts,amsmath,amssymb,amsbsy,amsthm}
\usepackage{bm}
\usepackage{color}
\usepackage{float}
\usepackage{hyperref}
\usepackage{mathrsfs}
\usepackage{longtable}
\usepackage[top=3cm,bottom=3cm,left=3cm,right=3cm]{geometry}
\usepackage{epstopdf}
\usepackage{mathabx}
\usepackage{stmaryrd}
\usepackage{ulem}
\usepackage{scalerel}
\usepackage{enumitem}
\usepackage{subfigure}
\usepackage{environments}
\numberwithin{theorem}{section}
\numberwithin{equation}{section}
\def\Z{\mathbb{Z}}
\def\R{\mathbb{R}}
\def\P{\mathbb{P}}
\def\N{\mathbb{N}}
\def\SS{\mathbb{S}}

\def\Es{\mathcal{E}}
\def\Fs{\mathcal{F}}
\def\As{\mathcal{A}}
\def\Bs{\mathcal{B}}
\def\Ts{\mathcal{T}}
\def\Os{\mathcal{O}}

\def\<{\langle}

\def\>{\rangle}
\def\argmin{{\rm argmin}}
\def\~{\tilde}
\def\^{\hat}
\def\-{\bar}
\def\per{{\rm{per}}}
\def\a{{\rm{a}}}
\def\hoc{{\rm{hoc}}}
\def\cb{{\rm{cb}}}

\def\cut{{\rm{cut}}}
\def\loc{{\rm{loc}}}

\def\conv{{\rm{conv}}}
\def\ol{\Omega^{\Lambda}}
\def\1/2{\frac{1}{2}}
\def\eps{\varepsilon}

\def\L{\Lambda}

\def\Uc{\mathscr{U}}
\def\Us{\mathcal{U}}
\def\Rc{\mathscr{R}}
\def\Yc{\mathscr{Y}}
\def\Kc{\mathscr{K}}
\def\Ks{\mathcal{K}}

\definecolor{lzcol}{rgb}{0.9, 0, 0}
\definecolor{hwcol}{rgb}{0, 0.9, 0}
\definecolor{ywcol}{rgb}{0, 0, 0.9}
\definecolor{todocol}{rgb}{1.0, 0.08, 0.58}

\begin{document}

\title[Analysis of a Higher Order Continuum Model]{A Priori Analysis of a Higher Order Nonlinear Elasticity Model for an Atomistic Chain with Periodic Boundary Condition}
% Short title for running heads:
%\shorttitle{Analysis of a Higher Order Continuum Model}

\author{Yangshuai Wang}
\address{Yangshuai Wang \\ School of Mathematical Sciences, MOE-LSC, and Institute of Natural Sciences \\
Shanghai Jiao Tong University \\ 800 Dongchuan Road \\ Shanghai \\ 200240 \\ China}
\email{yswang2016@sjtu.edu.cn}

\author{Hao Wang}
\address{Hao Wang \\  School of Mathematics \\ Sichuan University \\
  No.24 South Section One, Yihuan Road \\ Chengdu \\ 610065 \\ China}
\email{wangh@scu.edu.cn}

\author{Lei Zhang}
\address{Lei Zhang \\ School of Mathematical Sciences, MOE-LSC, and Institute of Natural Sciences \\
Shanghai Jiao Tong University \\ 800 Dongchuan Road \\ Shanghai \\ 200240 \\ China}
\email{lzhang2012@sjtu.edu.cn}

%\author{%
%{\sc
%Yangshuai Wang\thanks{Email: yswang2016@sjtu.edu.cn},
%and
%Lei Zhang\thanks{Co-Corresponding author. Email: lzhang2012@sjtu.edu.cn}} \\[2pt]
%School of Mathematical Sciences, MOE-LSC, and Institute of Natural Sciences \\
%Shanghai Jiao Tong University, 800 Dongchuan Road, Shanghai, 200240, China\\[6pt]
%{\sc and}\\[6pt]
%{\sc Hao Wang}\thanks{Corresponding author. Email: wangh@scu.edu.cn}\\[2pt]
%  School of Mathematics, Sichuan University, \\
%  No.24 South Section One, Yihuan Road, Chengdu, 610065, China
%}
% Short list of authors for running heads:
%\shortauthorlist{B. F. Nielsen \emph{et al.}}
%\shortauthorlist{Y. Wang, H. Wang and L. Zhang}

\date{\today}

\begin{abstract}
% Body of abstract:
{
Nonlinear elastic models are widely used to describe the elastic response of crystalline solids, for example, the well-known Cauchy-Born model. While the Cauchy-Born model only depends on the strain, effects of higher order strain gradients are significant and higher order continuum models are preferred, in various applications such as defect dynamics and modeling of carbon nanotubes. In this paper, we rigorously derive a higher order nonlinear elasticity model for crystals from its atomistic description in one dimension. We show that, compared to the second order accuracy of the Cauchy-Born model, the higher order continuum model in this paper is of fourth oder accuracy with respect to the interatomic spacing in the thermal dynamic limit. In addition, we discuss the key issues for the derivation of higher order continuum models in more general cases. The theoretical convergence results are demonstrated by numerical experiments. 
}
%% Keywords:
%\keywords
%{
%atomistic models, higher order strain gradient, higher order continuum models, modeling error, stability
%}
\end{abstract}

\maketitle

\section{Introduction}
\label{sec:introduction}
%
%\todo{Add motivations here, in the view of a/c coupling model.}
Nonlinear elasticity models are widely used to describe the elastic response of crystalline materials. The Cauchy-Born model is probably the most well known nonlinear elasticity model which is consistent with the atomistic theory of crystals, and it is second order accurate with respect to the atomistic model under certain technical assumptions \cite{MB_KH_Crystal_Latices, BLBL:arma2002, cdkm06, E:2007a, OrtnerTheil2012}. The Cauchy-Born energy density only depends on the strain, and can be interpreted as 'the stored energy per unit volume under a macroscopically homogeneous deformation equals the energy per unit volume in the corresponding homogeneous crystal' \cite{OrtnerTheil2012}.

The Cauchy-Born model is sufficiently accurate if the strain gradient is small. In various situations, nonlinear elastic models of higher order strain gradients are preferred. For example, the higher order strain gradients have significant impact for the defect zone \cite{Lyu2017Multiscale}, for curved crystalline sheets such as carbon nanotube \cite{Guo2006Mechanical}, and for the wave propagation in crystals \cite{arndt2005derivation, peridynamics2009}. 

In the mechanics literature, higher order continuum models were first derived in \cite{Triantafyllidis1993}, where two different approaches were used to derive the continuum macro model from the discrete micro model. Later on, it was found that higher order continuum models can efficiently capture the inhomogeneous deformation of the underlying crystal \cite{Sunyk2003On} and the curvature effect of carbon nanotubes  \cite{sun2008application}. In \cite{peridynamics2009}, the higher order gradient model links the atomistic model and nonlocal models such as peridynamics, and is able to capture the correct dispersive behavior of crystals. Recently, \cite{Lyu2017Multiscale} proposed a multiscale crystal defect dynamics (MCDD) model by adopting different higher order Cauchy-Born models (up to four) to construct atomistic informed constitutive relations for various defect process zones, and developed a hierachical strain gradient based finite element formulation.

%\cite[]{khoei2017computational}) 
On the contrary, only few works have been devoted to the mathematical analysis of higher order continuum models. 
In \cite{BLBL:arma2002}, it was proved that a higher order continuum model indeed has better accuracy in terms of energy. 
It was shown in \cite{arndt2005derivation} that the higher order continuum models in \cite{Triantafyllidis1993} might be ill-posed, and may lead to an uncontrolled behavior of the solution. The so-called "inner expansion", which is based on the formal Taylor expansion of the deformation gradient at some carefully chosen expansion points, was proposed to derive continuum models from the atomistic models (with pair interactions) and minimize the remainder terms of the energy.  While a well-posed higher order continuum approximation was first developed in \cite{arndt2005derivation}, a rigorous mathematical analysis was not included.

%There are many works for elastodynamics problems using higher order derivatives of the deformation gradient \cite{Zabusky1964, Rose1986, Kevre2002}.  A systematic way to derive higher order terms is given by \cite{Zabusky1964} and was further developed by \cite{Rose1986}, which used the Taylor expansion of the discrete evolution equation. Numerical experiments have been given by \cite{Kevre2002}. 

The main purpose of the current work is to derive a higher order nonlinear elasticity model from the atomistic model (with pair potential) in one dimension and present a rigorous a priori analysis of the obtained higher order model. The derivation of the model essentially follows the techniques of the 'inner expansion' introduced in \cite{arndt2005derivation} which leads to a well-posed higher order  continuum model.

The major contribution of the current work is that, to the best knowledge of the authors, it for the first time provides a rigorous analysis and error estimate for the energy minimizer of the higher order continuum model and numerically demonstrate the fourth order accuracy for such a model. To be precise, we will show that the approximation error, which is also known as the modeling error, is of $O(\eps^4)$. Namely, if we scale the system by in interatomic spacing $\eps$, we have
$$\|\nabla u^\a-\nabla u^\hoc\|_{L^2} \le C \eps^4,$$ 
where $u^\a$ and $u^\hoc$ are the solutions to the atomistic model and the higher order continuum model, respectively, which will be defined in Section \ref{Sec: variational problems}, and the constant $C$ depends only on some higher order partial derivatives of the interatomic potential $\phi$ and on the regularity of $u^\a$. We essentially extend the analytical framework in \cite{OrtnerTheil2012} to the higher order continuum model in one dimension with pair interactions, which include the analysis of the modeling error, stability and convergence estimates. In addition, we point out the possibilities and challenges to extend those results to the physically more relevant cases of multibody interactions and of higher dimensions in Section \ref{sec:conclusion}.
%To keep the presentation clear and simple but retain the most important features of the problem, we assume that the atoms are interacted though a pair potential with finite interaction range. We impose the periodic boundary condition on the atomistic system, to avoid unnecessary difficulty with the boundary conditions which may prevent us from observing the correct convergence rate in strain. 

%For the derivation of the model, we essentially follows the techniques of "inner expansion" in \cite[]{arndt2005derivation} and the Euler-Lagrange equation of the resulting continuum model is thus a sixth order elliptic equation. For the a priori analysis, we adopt the analytical framework proposed in \cite[]{OrtnerTheil2012} to our higher order continuum model and show that the model is of fourth order accuracy compared with underlying atomistic model. We demonstrate the theoretical convergence estimate with numerical experiments. Possible extensions to higher dimensions and to general multi-body interactions will be discussed. 

%\commentlz{higher dimension?}
%
%\commentlz{multibody?}

\subsection{Outline}
The paper is organized as follows. 

\S~\ref{sec:formulation:int} is a preliminary section with some interpolation results for lattice functions. Those results are the extensions of similar results in \cite{AlexIntp:2012} to higher order interpolations, and they will be used extensively and play a key role in the forthcoming analysis. 

In \S~\ref{sec:modelset}, we set up the atomistic model and its continuum approximations. In particular, we derive the general formulation of the higher order continuum approximation. We carry out the analysis of the modeling error in \S~\ref{sec:modelingerror}, and the analysis of the stability in \S~\ref{sec:stability}, respectively, for the higher order continuum model.

Our main result Theorem \ref{priori} in \S~\ref{sec: a priori error estimate} states that, the higher order continuum model (depending on the strain $\nabla u$ and the second order derivative of the strain $\nabla^3 u$) which we derived from the original atomistic model has fourth order accuracy.  Numerical experiments in \S~\ref{sec:numerics}  complement and justify the theoretical analysis. We make concluding remarks and point out some promising directions for future work in \S~\ref{sec:conclusion}. %Some auxiliary results and lemmas are given in \S~\ref{sec:proof}.

\subsection{Summary of notations}
We denote the directional derivative in the direction $\rho$ by $\nabla_{\rho}f(x) = \rho \cdot \nabla f(x)$. Let the symbol $\<\cdot, \cdot\>$ denote the duality pairing,  the first and second variations of $\Fs$ at $u$ are defined as 
$$\<\delta\Fs(u), v\> := \lim_{t\rightarrow0} t^{-1}(\Fs(u+tv) - \Fs(u)) \qquad $$
$$ \text{and}\  \<\delta^2\Fs(u)v, w\> := \lim_{t\rightarrow0} t^{-1}(\delta\Fs(u+tw) - \delta\Fs(u),v).$$
We use the convention that '$\lesssim$' stands for '$\leq C$', where $C$ is a generic constant that does not depend on the strain and its higher order derivatives.
%may not depend on any data in the model, nor on any functions involved in the inequality \commentlz{vague statement}. 

% \input{2_Preliminary_Results}
\section{Preliminary Results for the Interpolation of Lattice Functions}
\label{sec:formulation:int}

We define the reference lattice as $\Lambda := \Z$. The space of lattice functions is given by
\begin{equation}
\Uc:=\left \{v:\Lambda \rightarrow \R \right \}.
\end{equation}

The atomistic model is defined over the lattice functions, while the continuum model is defined over continuous functions. We will introduce certain interpolations to bridge lattice functions and continuous functions on the real line, which is the adaptation of the results in \cite{AlexIntp:2012} to higher order interpolations.

For a lattice function $v\in\Uc$, we define the finite difference,
\begin{equation}
D_{\rho}v(\xi) := v(\xi+\rho) - v(\xi),  \quad \text{for} \hspace{0.2cm} \xi \in \Lambda, \rho \in \Rc,
\label{eq:finitediff}
\end{equation}
where the finite set $\Rc\subset \{1, 2, ...\}$ represents the interaction range of the atomistic model. It is easy to observe that $D_{-\rho}v(\xi) = -D_{\rho}v(\xi-\rho)$. 

For simplicity, we only consider periodic boundary condition in this paper, and we limit our analysis in the periodic domain $\Omega = [-N, N]$ for a fixed $N \in \N$. We also denote $\Omega \cap \Lambda$ by $\ol$ .

In the following sections, we introduce three different types of interpolation for a lattice function $v \in \Uc$.

\subsection{Interpolation based on finite differences}
\label{sec: interpolation based on finite differences}
%Since we are primarily interested in the a priori analysis of the higher order continuum model, we need smooth representations of lattice functions to which the solutions to the atomistic model belong. A 
%natural measure of the local smoothness of a lattice function $v\in\Uc$ would be the higher order finite differences, which are, however, cumbersome and of little use to our analysis. Instead, we introduce smooth interpolations based on the finite differences of different orders whose derivatives provide equivalent information. The construction of such interpolation operator is not unique, one example would be (c.f. Section 3.1, \cite{LuOr:acta}) $I: \Uc \rightarrow C^{2}$ such that $Iv$ is the unique continuous function such that $Iv|_{(\xi-1,\xi)}$ is a polynomial of degree 5 and for all $\xi\in \Lambda$
%\begin{align}\label{definI}
%Iv(\xi) &:= v(\xi), \\ \nonumber
%\nabla Iv(\xi) &:= \1/2\big[D_{1}v(\xi)-D_{-1}v(\xi)\big], \\ 
%\nabla^2 Iv(\xi) &:= D_{1}v(\xi) + D_{-1}v(\xi). \nonumber
%\end{align}
%If we consider $v$ as a sufficiently smooth function, say $v \in C^4$, instead of a lattice function (which we usually do in the analysis), it can be easily seen that the first and second derivatives of $Iv$ equals to the second order finite difference of $v$ at $\xi \in \Lambda$. Higher order finite difference operators will be defined in subsequent analysis where needed.

Since we are primarily interested in the a priori analysis of the higher order continuum model, we need smooth interpolates of the lattice functions which include the solutions to the atomistic model. A 
natural measure of the local smoothness of a lattice function $v\in\Uc$ would be the higher order finite differences, which are, however, cumbersome and of little use to our analysis. Therefore we define a smooth interpolation operator $\Pi : \Uc \rightarrow C^4$ such that $\Pi v \in W^{5,2}$ \cite[Theorem 5.2]{Braess2007} which is an Hermitian interpolation of degree 9, based on fourth order approximations of derivatives up to fourth order. One example of such interpolation could be defined as follows: for a lattice function $v \in \Uc$, we let 
\begin{align}
\ \Pi v(\xi) &:= v(\xi), \nonumber \\
\nabla \Pi v(\xi) &:= \frac{1}{12}\big[-D_{2}v(\xi)+8D_{1}v(\xi)-8D_{-1}v(\xi)+D_{-2}v(\xi)\big], \nonumber \\
\nabla^2 \Pi v(\xi) &:= \frac{1}{12}\big[-D_{2}v(\xi)+16D_{1}v(\xi)+16D_{-1}v(\xi)-D_{-2}v(\xi)\big], \nonumber \\
\nabla^3 \Pi v(\xi) &:= \frac{1}{8}\big[-D_{3}v(\xi)+8D_{2}v(\xi)-13D_{1}v(\xi)+13D_{-1}v(\xi)-8D_{-2}v(\xi)+D_{-3}v(\xi)\big], \nonumber \\
\nabla^4 \Pi v(\xi) &:= \frac{1}{6}\big[-D_{3}v(\xi)+12D_{2}v(\xi)-39D_{1}v(\xi)-39D_{-1}v(\xi)+12D_{-2}v(\xi)-D_{-3}v(\xi)\big],
%\nabla^5 \Pi v(\xi) &= \frac{1}{8}\big[-D_{3}v(\xi)+8D_{2}v(\xi)-13D_{1}v(\xi)+13D_{-1}v(\xi)-8D_{-2}v(\xi)+D_{-3}v(\xi)\big].
\label{eq: 5th order finite difference}
\end{align}
where $\xi \in \L$. Such interpolations will be used in the analysis of the modeling error  since it satisfies both the requirement of smoothness and certain equivalence with the other two types of interpolations (c.f. Proposition \ref{prop: regularity of Pi v} and Proposition \ref{prop: stability of tildev}).

\subsection{Interpolation based on nodal basis functions}
\label{sec: interpolation based on nodal basis functions}

Let $\zeta \in W^{3, \infty}$ and $\zeta(\cdot - \xi)$ is a basis function associated with the lattice site $\xi$. We assume that $\zeta$ has a compact support and $\int_{\R}\zeta dx = 1$. We also assume that the discrete convolution with the basis function $\zeta$ preserves cubic functions, namely  
\begin{equation}
\label{eq: zeta preserving 3rd order polynomials}
\sum_{\xi\in\Lambda}(a+b\xi+c\xi^2+d\xi^3)\zeta(x-\xi)=a+bx+cx^2+dx^3, \qquad \forall a,b,c,d\in\R.
\end{equation}
One possible construction of such $\zeta$ can be the cubic spline basis function \cite[Section 3.2]{hollig2003finite}.
%We refer to \cite[]{hollig2003finite} \commentlz{page} for one possible construction of such $\zeta$ %\commentlz{it is better to include it here if not too complicated}.

We then define the standard interpolation as follows, 
\begin{equation}
\^v(x):=\sum_{\xi \in \Lambda}v(\xi)\zeta(x-\xi),\qquad \text{for} \hspace{0.2cm}  v \in \Uc.
\label{eq:higherintp}
\end{equation}
%It is natural to see that we could choose $\^v$ be the cubic spline interpolation of $v$ in one dimensional %case, and it can be easily transferred into \ref{eq:higherintp}. 
%This interpolation need some careful assumptions to obtain some important properties. 

%\begin{remark}
%We refer to  for the construction of such nodal basis functions $\zeta$.
%\end{remark}

\subsection{Interpolation based on convolution}
\label{sec: interpolation based on convolution}

The third interpolation can be constructed by the convolution of the nodal basis interpolation $\^v$ with $\zeta$ (c.f. \cite{OrtnerTheil2012}):
\begin{equation}
\~v(x):=(\zeta*\^v)(x)=\int_{\R}\zeta(x-y)\^v(y)dy.
\label{eq:higherintpcon}
\end{equation} 
We note that $\~v$ is in fact a quasi-interpolant of the lattice function $v$ since $\~v|_{\Lambda} \neq v|_{\Lambda}$ in general . The purpose of introducing $\~v$ is to construct the atomistic stress tensor, which will be defined in Section \ref{sec:modelingerror}. The quasi-interpolation $\~v$ leads to the so-called localization formula \cite{OrtnerTheil2012}
\begin{align}
\label{eq:localization3}
D_{\rho}\~v(\xi)&=\int_0^1\nabla_{\rho}\~v(\xi+t\rho)dt \nonumber \\ 
&=\int_{\R}\int_0^1\zeta(\xi+t\rho-x)dt\cdot \nabla_{\rho}\^v(x)dx   \nonumber \\
&=: \int_{\R}\chi_{\xi,\rho}(x)\cdot \nabla_{\rho}\^v(x)dx.
\end{align}
With the help of \eqref{eq:localization3} we are able to rewrite the finite differences of test functions in an integral form.

\subsection{Properties of the interpolation functions}

We show the regularity and stability (in $W^{1,2}$ seminorm) of the convolution based interpolant $\~v$ and the finite difference based interplant $\nabla \Pi v$ in the following two propositions. The $W^{1,2}$ stability with respect to the nodal basis interpolant $\^v$ shows that those three interpolations are essentially equivalent. The first proposition is similar to Proposition 3.1 of \cite{LuOr:acta} and we follow the same lines of proof, while the second proposition can be found in \cite{AlexIntp:2012}.
%\alertyw{The first proposition is just Proposition 3.1 of \cite{LuOr:acta}  and the proof of the second proposition is given in the appendix.} 

\begin{proposition}
\label{prop: regularity of Pi v}
%(Stability of $Iv$)~~~
Let $v\in\Uc$ and $\Pi : \Uc \rightarrow C^4$ be a smooth interpolation operator satisfying \eqref{eq: 5th order finite difference}.
%\begin{equation}
%\nabla Iv(\xi) = \1/2\big[D_{1}v(\xi)-D_{-1}v(\xi)\big].
%\end{equation}
We then have
\begin{equation}
||\nabla \^{v}||_{L^2} \lesssim ||\nabla \Pi v||_{L^2} \lesssim ||\nabla \^{v}||_{L^2}, \quad \forall  v \in \Uc.
\end{equation}
\end{proposition}
\begin{proof}
The first inequality follows from H\"older's inequality, the observation that $\int_{\xi}^{\xi+1}\nabla\Pi v(x) dx = v(\xi+1) - v(\xi)$, and Theorem 2 in \cite{AlexIntp:2012}. The second inequality holds by Lemma 5.4 of \cite{bqce}.
\end{proof}

\begin{proposition} 
\label{prop: stability of tildev}
%(Regularity and Stability of $\~v$)~~~
Let $v\in\Uc$. The nodal basis interpolation $\^v$ of $v$ belongs to $W_{loc}^{3, \infty}$ and the convolution based interpolation $\~v$ belongs to $W_{loc}^{5, \infty}$. Moreover, the interpolants $\^v$ and $\~v$ have the following norm equivalence 
\begin{equation}\label{stab tilde v}
||\nabla \^{v}||_{L^2} \lesssim ||\nabla \~v||_{L^2} \lesssim ||\nabla \^{v}||_{L^2}, \quad \forall  v \in \Uc.
\end{equation}
\end{proposition}
\begin{proof}
The regularity of $\^v$ and $\~v$ follows from \cite[Lemma 1]{AlexIntp:2012} and the stability of $\nabla \~v$ with respect to $\nabla \^v$ follows from \cite[Theorem 2]{AlexIntp:2012} since $W^{3, \infty} \subset W^{1, \infty}$.
\end{proof}

\section{Atomistic Model and its Continuum Approximations}
%
%In this section, we introduce the atomistic model and the continuum approximations we consider.

\label{sec:modelset}
%
%the discrete space that deformation $y$ and displacement $u$ lie in is $\Uc:=\left \{v:\Lambda \rightarrow \R \right \}$.
%We will firstly give a brief review of atomistic model we considered in \S~\ref{sec:formulation:atm} and the interpolation of discrete function in \S~\ref{sec:formulation:int}. The space of admissible displacements and some basic properties of interpolate functions will also be introduced in \S~\ref{sec:formulation:space} and in \S~\ref{sec:formulation:prop}.

\subsection{The atomistic model}
\label{sec:formulation:atm}
In this section, we introduce the atomistic model as the ground truth description of the atomistic system. We impose a periodic boundary condition on the atomistic system, to avoid unnecessary technical difficulty which may prevent us from observing the correct convergence rate. For example, if a Dirichlet boundary condition is enforced, a suitable boundary layer of "ghost atoms" should be added in order to guarantee the higher order convergence rate. For boundary value problems for Cauchy-Born model, please refer to 
\cite{Braun_2016_BC_for_AC}. 

Fix $N \in \N$, and define the space of $2N$-periodic mean zero displacements as
\begin{equation}
  \Uc^{\per} = \big\{ u \in \Uc:
  u_{\xi+2N} = u_{\xi}, {\textstyle \sum_{\xi \in \ol}} u_\xi = 0  \big\}.
  \label{def: Displacement Space}
\end{equation}
The set of admissible deformations is given by
\begin{equation}\label{def: DeformationSet}
\Yc^{\per} = \big\{y \in \R^\Z: y_\xi = F \xi + u_\xi, u \in \Uc^{\per} \big\},
\end{equation}
where $F > 0$ is a macroscopic deformation gradient.

The atomistic energy (per period) at a deformation $y \in \Yc^{\per}$ is defined by
\begin{equation}
\Es_{\a}(y)=\sum_{\xi \in \ol}\sum_{\rho \in \Rc}\phi(D_{\rho}y(\xi)),
\label{Aenergyiny}
\end{equation}
where $\phi \in C^3((0, +\infty];{\R})$ (for example, a Lennard-Jones or Morse type potential). Using the relationship described in \eqref{def: DeformationSet} and with a slight abuse of notation, the energy can be rewritten in the form of the displacement $u \in \Uc^{\per}$ 
\begin{equation} 
\Es_{\a}(u)=\sum_{\xi \in \ol}\sum_{\rho \in \Rc}\phi_{\rho}(D_{\rho}u(\xi)),
\label{Aenergyinu}
\end{equation}
where $\phi_{\rho}(r)=\phi(r+F\rho)$ is the Lennard-Jones or Morse potential under macroscopic deformation gradient $F$. We will use \eqref{Aenergyinu} as the energy functional of our atomistic model throughout the paper.

%Let $u \in \Uc$ be the displacement of an infinite lattice and the corresponding deformation $y$ be defined by
%\begin{equation}
%y = F\xi + u, \label{def: DeformationSet}
%\end{equation}
%where $F$ is a macroscopic deformation gradient. We can define the atomistic energy under $y$ to be
%\begin{equation}
%\Es_\a(y)=\sum_{\xi \in \ol}\sum_{\rho \in \Rc}\phi(D_{\rho}y(\xi)),
%\end{equation}
%where $\phi(r)$ is a Lennard-Jones or Morse type interaction and $\phi \in C^3((0, +\infty];{\R})$. Using the relationship described in \ref{def: DeformationSet} and with a slight abuse of notation, the energy can be rewritten in terms of the displacement $u \in \Uc$
%\begin{equation} 
%\Es_\a(u)=\sum_{\xi \in \ol}\sum_{\rho \in \Rc}\phi_{\rho}(D_{\rho}u(\xi)),
%\label{Aenergyinu}
%\end{equation}
%where $\phi_{\rho}(r)=\phi(r+F\rho)$ is the modified Lennard-Jones potential. 
%We will use \ref{Aenergyinu} as the energy functional of our atomistic model through out the paper.

%We note that since $\Lambda$ is an infinite lattice the energy functional \ref{Aenergyinu} may be infinite. 
%To make the problem well defined, we introduce two discrete Sobolev spaces of atomistic displacements using the interpolation of lattice function. 
We should also equip the space $\Uc^{\per}$ with $L^2$-norm using the interpolation of lattice function to obtain
\begin{align} 
&\Uc^{1, 2}:=\{u\in\Uc^{\per} \big| ||\nabla \^u||_{L^2(\Omega)}<+\infty\}.
%&\Us_0^{1, 2}:=\{u\in\Uc^{1, 2} \hspace{0.2cm}\big|\hspace{0.2cm} \text{supp}(\^{u}) \hspace{0.1cm} \text{is compact}\}.
\end{align}
%It can be proved that $\Us^{1,p}$ is a Banach space for $p\in[1,\infty]$ and $\Us_0^{1,p}$ is dense in $\Us^{1,p}$ (c.f. \cite{AlexIntp:2012}). It is also shown in \cite{2013-defects, LuOr:acta} that for $u \in \Uc^{1,p}$, we have $|u(\xi)| \ll |\xi|$ as $\xi \rightarrow \infty$. Therefore, the discrete deformation induced by $u$ satisfies the far-field boundary condition
%$$y(\xi) = F\xi + o(|\xi|) \qquad \text{as} \quad \xi \rightarrow \infty.$$
We equip the space $\Uc^{1,2}$ with the norm $||u||_{\Uc^{1,2}} := ||\nabla \^u||_{L^2(\Omega)}, \forall u \in \Uc^{1,2}$.

On the other hand, to exclude arbitrarily large deformations which are not covered by our results, we place an $L^{\infty}$-bound on the displacement gradient and define
\begin{align}
\Kc := \{ u\in\Uc^{1,2} \big| |D_{\rho}u(\xi)| \leq\kappa, \forall \xi \in \ol, \rho \in \Rc \},
\end{align}
where $\kappa > 0$ is a fixed constant. One reasonable choice of $\kappa$ is $\frac{1}{4}F$ which could be ensured through conditions on the external force. 

It is straightforward to see that $\Es_\a$ is well defined in $\Uc^{1,2} \cap \Kc$ (c.f. \cite[Theorem 1]{OrtnerTheil2012}),  which will be the solution space of the atomistic variational problem that will be introduced in Section \ref{Sec: variational problems}.

%To avoid technical difficulties with boundaries we formulate a model problem with periodic boundary conditions. Let $N \in \N$ be a finite number and define the space of $2N$-periodic zero mean displacements as
%\begin{equation}
  %\Uc^{\per} = \big\{ u \in \Uc:
  %u_{\xi+2N} = u_{\xi}, {\textstyle \sum_{\xi = -N}^N} u_\xi = 0  \big\}.
  %\label{def: Displacement Space}
%\end{equation}
%The set of admissible deformations is given by
%\begin{equation}
%\Yc^{\per} = \big\{y \in \R^\Z: y_\xi = F \xi + u_\xi, u \in \Uc \big\},
%\label{def: DeformationSet}
%\end{equation}
%where $F > 0$ is a macroscopic deformation gradient.

%
% The relationship between deformation and displacement is $y(\xi) = u(\xi) + F\xi$, where $\xi$ denotes the position of the atom and $F$ is the macroscopic stretch.

%From \cite{OrtnerTheil2012}, it is easy to see that if we choose $\kappa$ to be sufficiently small, then displacements belonging to $\Kc$ give rise to injective deformations. We may also equip the space $\Kc$ with an $L^2$-norm on the interpolation of its element and define
%\begin{align} 
%&\Us^{1, 2}:=\{u\in\Kc \hspace{0.2cm}\big|\hspace{0.2cm} ||\nabla \^u||_{L^2(-N,N)}<+\infty\}.
%\end{align}

Finally, we assume the decay hypothesis of the derivatives of the interaction potential $\phi_{\rho}$ \cite{OrtnerTheil2012, LuOr:acta}, which is a crucial ingredient in our analysis in Section \ref{sec:modelingerror}. For $\rho \in \Rc, 1 \leq j \leq k$, we require
\begin{equation}
M^{(j, s)}:= \sum_{\rho\in\Rc}m^{(j, s)}(\rho) < \infty,
\label{eq: boundedness of the potential}
\end{equation}
where 
$$m^{(j, s)}(\rho) := \rho^{j+s} \sup_{g}|\phi^{(j)}_{\rho}(g)|,$$
where $\phi_{\rho}^{(j)}$ denotes the $j$th derivative of $\phi_{\rho}$. This will ensure that $\Es_\a$ is $k$ times Fr\'echet differentiable. 
%We also require for $2 \leq j \leq k$ and $\rho \in \Rc$ that 
%$$M_s^{(j)}:= \sum_{\rho\in\Rc}m^{(j)}_s(\rho)< \infty,$$
%where 
%$$m_s^{(j)}(\rho) :=  m^{(j)}(\rho)\rho^{4}.$$
We note that in the current work $k\leq 6$ and the interaction range $\Rc=\{1, 2, ..., r_{\cut}\}$ is finite. %We also remark that $M^{(j)} \leq M_s^{(j)}$.

\begin{remark}
We note that the assumption \eqref{eq: boundedness of the potential} does not hold at $0$, for example, for Lennard-Jones or Morse potential. However for all practical purposes, we are concerned with configurations not far from reference configuration, and the atoms will not get accumulated. Therefore, it is reasonable to make the assumption \eqref{eq: boundedness of the potential}.
\end{remark}

%\subsection{The space of admissible displacements}
%%
%\label{sec:formulation:space}
%%
%In this subsection, we define the space which the atomistic displacements lie in. We introduce two spaces of atomistic displacements by using the interpolation of lattice function.
%\begin{align} 
%&\Us^{1, 2}:=\{u\in\Uc \hspace{0.2cm}\big|\hspace{0.2cm} ||\nabla \^u||_{L^2}<+\infty\}, \\ \nonumber
%&\Us_0^{1, 2}:=\{u\in\Uc \hspace{0.2cm}\big|\hspace{0.2cm} \text{supp}(\^{u}) \hspace{0.1cm} \text{is compact}\}.
%\end{align}
%The rigorous proof can be found in \cite{AlexIntp:2012} to see that $\Us^{1,p}$ is a Banach space for $p\in[1,\infty]$ and $\Us_0^{1,p}$ is dense in $\Us^{1,p}$. We denote them as discrete Sobolev spaces.
%
%Let $u \in \Us^{1,p}$, be a discrete displacement and $y(\xi) :=  u(\xi) + F\xi$ be the associated deformation. It is shown in \cite{2013-defects, LuOr:acta} that $|u(\xi)| \ll |\xi|$ as $\xi \rightarrow \infty$. Therefore, the discrete deformation should satisfy the far-field boundary condition
%$$y(\xi) = F\xi + o(|\xi|) \qquad \text{as} \quad \xi \rightarrow \infty.$$

%\subsection{The properties of interpolation functions}
%\label{sec:formulation:prop}

%\section{Higher order continuum model}
%\label{sec:derivation}

\subsection{The continuum approximations}
\label{sec:continuum}
We introduce the continuum approximations of the atomistic model in this section. There are a number of approaches to obtain such approximations \cite{Triantafyllidis1993, BLBL:arma2002, arndt2005derivation, E:2007a}. We adopt the inner expansion technique in \cite{arndt2005derivation}, which can easily satisfy the energy consistency and leads to a well-posed (the precise meaning of well-posedness will be made clear in Remark \ref{posedness}) higher order continuum model which depends on the first order and third order derivatives of $u$.

To introduce the continuum approximation, we begin with the atomistic model 
$$ \Es_\a(u)=\sum_{\xi \in \ol}\sum_{\rho \in \Rc}\phi_{\rho}(D_{\rho}u(\xi)),$$
which is written as a sum of the site energy. Though we only require $u\in\Uc^{1,2}$ in the atomistic model, by Section \ref{sec:formulation:int}, we can replace $u$ by its proper smooth interpolation, for example, the finite difference based interpolation $\Pi u$, without changing the atomistic energy. When no confusion occurs, we identify the discrete lattice function $u$ with its smooth interpolation in the following derivation.
 
After taking the Taylor expansion of the site energy $D_{\rho}u(\xi)$ at the midpoints $\xi' := \frac{\xi+\xi+\rho}{2}$ of the bonds $(\xi, \xi+\rho)$ and truncating at order three, we have
\begin{align}
D_{\rho}u(\xi) =& u(\xi+\rho) - u(\xi)  \nonumber \\
=& \Big[u(\xi') + \frac{\rho}{2}\nabla u(\xi') + \frac{1}{2}\nabla^2u(\xi')(\frac{\rho}{2})^2 +  \frac{1}{6}\nabla^3u(\xi')(\frac{\rho}{2})^3 + ...\Big] \nonumber  \\ 
&- \Big[u(\xi') - \frac{\rho}{2}\nabla u(\xi') + \frac{1}{2}\nabla^2u(\xi')(\frac{\rho}{2})^2 - \frac{1}{6}\nabla^3u(\xi')(\frac{\rho}{2})^3 + ...\Big] \nonumber \\ 
\approx& \rho\nabla u(\xi')+\frac{\rho^3}{24}\nabla^3u(\xi').
\label{eq: Taylor expansion of u}
\end{align}
The atomistic model can then be approximated as
\begin{equation}
\Es_\a(u) \approx \sum_{\xi\in \ol}\sum_{\rho\in \Rc}\phi_{\rho}( \rho\nabla u(\xi')+\frac{\rho^3}{24}\nabla^3u(\xi')).
\end{equation}
An approximation step similar to the Riemann sum leads to the following higher order continuum (HOC) approximation
\begin{equation}
\Es_\hoc(u)=\int_{\Omega}\sum_{\rho\in \Rc} \phi_{\rho}(\rho\nabla u+\frac{\rho^3}{24}\nabla^3u)dx,
\label{hocmodel}
\end{equation}
where $\phi_{\rho}$ is the Lennard-Jones or Morse potential under macroscopic deformation gradient $F$ introduced in Section \ref{sec:formulation:atm} and $\nabla u = \nabla u(x)$ is the gradient of $u$ with respect to $x$. 

The well-known Cauchy-Born approximation can be obtained by preserving only the first order term in \eqref{hocmodel} 
\begin{equation}
\label{cbmodel}
\Es_{\cb}(u):=\int_{\Omega} \sum_{\rho\in \Rc }\phi_{\rho}(\rho\nabla u) dx.
\end{equation}

We note that the higher order energy functional defined in \eqref{hocmodel} depends both on $\nabla u$ and $\nabla^3 u$, whereas the Cauchy-Born energy functional \eqref{cbmodel} depends only on $\nabla u$.

%\begin{remark}
%In \cite[Section 5.3]{arndt2005derivation} they also simply discussed how to derive the higher order continuum model for many body potential. The derivation depends strongly on the interaction and the type of the potential. We do not have a general framework of the derivation of higher order continuum model so far. Since this paper focuses on the a priori analysis, we will present the more general result and details of the many body potential in our future work. 
%\end{remark}

\begin{remark}\label{posedness}
In \cite{Triantafyllidis1993} the authors derived two higher order continuum models which contain both $\nabla u$ and $\nabla^2 u$ terms from the atomistic model. However, it was discovered in \cite[Section 4]{arndt2005derivation} that these two higher order continuum models are ill-posed, which led to an uncontrolled behavior of the solution. For example, if we take the harmonic potential $\phi(r)=\1/2 (r-1)^2$ as the atomistic potential, one of the (ill-posed) higher order continuum models is
\begin{equation}\label{eq:illenergy}
\Es_{\rm ill}(u) = \int_{\Omega} \frac{1}{2}(\nabla u)^2 - \frac{1}{24}(\nabla^2 u)^2 dx,
\end{equation}
%\ys{which is the linear case of \eqref{hocmodel can}} 
and the corresponding Euler-Lagrangian equation of \eqref{eq:illenergy} is 
\begin{equation}\label{eq:illEL}
\nabla^2 u - \frac{1}{12}\nabla^4u = 0.
\end{equation}
We observe from \eqref{eq:illenergy} that the energy is not positive definite, and the differential operator on the left hand side of \eqref{eq:illEL} is not elliptic. \eqref{eq:illenergy} is ill-posed in this sense, which also means the energy \eqref{eq:illenergy} is not stable in the sense of \eqref{hocstab}.

It is possible to derive a well-posed higher order continuum model \eqref{hocmodel} through the inner expansion technique in \cite{arndt2005derivation}. For example, for the above harmonic potential case, we can obtain
\begin{equation}\label{eq:wellenergy}
\Es^{\rm lin}_{\rm hoc}(u) = \int_{\Omega} \frac{1}{2}(\nabla u)^2 - \frac{1}{24}(\nabla^2 u)^2 + \frac{1}{1152}(\nabla^3 u)^2dx,
\end{equation}
and the corresponding Euler-Lagrange equation of \eqref{eq:wellenergy} is 
\begin{equation}\label{eq:wellEL}
\nabla^2 u - \frac{1}{12}\nabla^4u +  \frac{1}{576}\nabla^6u= 0,
\end{equation}
which is actually a special case of the general (nonlinear) Euler-Lagrange equation \eqref{hocEL}. \eqref{eq:wellenergy} is well-posed, and thus is stable. 
\end{remark}

\subsection{The Variational problems}
\label{Sec: variational problems}

In this section, we define the variational problem for both atomistic model and higher order continuum models. 

\subsubsection{The Atomistic problem}
\label{sec:aprob}

We apply an external force $f$ to the atomistic system, in order to generate a nontrivial solution to the atomistic model. Following previous literature \cite{OrtnerTheil2012, E:2007a, OrtnerSuli:2008a, OrtnerWang:2011}, the external force of the atomistic model is modeled as a dead load so that the work of the external force is given by
\begin{align}
\<f, u\>_{\ol} := \sum_{\xi \in \ol}f(\xi)u(\xi), 
\label{oldext}
\end{align}
where $u \in \Uc^{1,2}$ is a displacement and $f|_{\Lambda}\in\Uc^{1,2}$.
The atomistic problem is: find a (local) minimizer $u^\a$ such that
\begin{align}
u^{\a} \in \argmin \Big\{ \Es_{\a}(u) - \<f, u\>_{\ol} \big| u \in \Uc^{1,2} \cap \Kc \Big\}.
\label{Aprob}
\end{align}
If $u^{\a}$ is a solution to \eqref{Aprob}, then it satisfies the first-order condition 
\begin{equation}
\<\delta \Es_{\a}(u^{\a}), v\> = \<f, v\>_{\ol}, \quad \forall v \in \Uc^{1,2}.
\label{Avarprob}
\end{equation}
We call a solution $u^{\a}$ of \eqref{Aprob} (strongly) stable if there exists $c_0>0$ such that
\begin{equation}
\<\delta^2 \Es_{\a}(u^{\a})v, v\> \geq c_0||\nabla v||^2_{L^2(\Omega)}, \quad \forall v \in \Uc^{1,2}.
\end{equation}

%However, it is easy to see that \ref{oldext} is trapzoid rule of \ref{cext}, which means it is only of second order accuracy. We will see in Section \ref{sec:numerics} that the quadrature error will dominate and prevent the observation of the high accuracy of our higher order continuum model. Therefore, we redefine the atomistic external energy from the continuous one to improve the order of the quadrature error.

%To this end, we first introduce the interpolation operator $I: \Uc \rightarrow C^2(\R)$ based on higher order finite differences whose definition is given later in Section \ref{sec: a priori error estimate} by \eqref{eq: 5th order finite difference}. Let $\Lambda* := \Lambda \cup \{\Lambda + \frac{1}{2}\}$. We then define the external energy of the atomistic problem to be
%\begin{align}
% \<f, u\>_{\Lambda*} = \sum_{\xi'\in\Lambda*}c(\xi')f(\xi')I u(\xi'),
% \label{aext}
%\end{align}
%where $c(\xi')$'s are the quadrature coefficients which lead to, say, the Simpson's integral rule. We remark that the Simpson integral rule is not the unique choice and the only requirement is that the order of the quadrature is not less than four. 

\subsubsection{The Higher order continuum problem}

To define the variational problem with respect to $\Es_\hoc$, we first introduce the following space
\begin{equation}
  \Us^{1,2} := \big\{ u \in W^{4,2}\cap W^{1,2}:
  \nabla^{j} u(x+2N) = \nabla^{j} u(x), j = 0,1,2,  {\textstyle \int_{\Omega} u dx = 0}  \big\}.
  \label{def: Displacement Space}
\end{equation}
%We will see in Section \ref{sec:A&C stresses} that the Euler-Lagrange equation associated with the higher order continuum model is a sixth order elliptic equation and thus has a natural solution space $W^{3, 2}$.

In order to apply the inverse function theorem (Lemma \ref{IFT}) to obtain the error estimate $\|\nabla u^\a-\nabla u^\hoc\|_{L^2}$, we equip the space $\Us^{1,2}$ with the $W^{1,2}$ norm 
$$||u||_{\Us^{1,2}}:=||\nabla u||_{L^2(\Omega)}, \quad \forall u \in \Us^{1,2}.$$ 
We denote $W^{-1,2}$ as the standard topological dual of $W^{1,2}$.

We define the $L^2$ inner product for $u,v \in \Us^{1,2}$,
\begin{equation}
\<u, v\>_{\Omega}=\int_{\Omega}u \cdot v dx.
\label{cext}
\end{equation}

Similar to $\Kc$, we shall assume that all displacement gradients satisfy a uniform bound. To that end we define
\begin{align}
\Ks := \{ u\in W^{3, \infty}_{\loc} \big| ||\nabla^3 u||_{L^{\infty}(\Omega)} \leq\kappa \},
\end{align}
where $\kappa$ is the same constant as in the definition of $\Kc$. Notice that $\Ks$ is an open set in $\Us^{1,2}$, and 
 $\Es_\hoc$ is well defined on $\Ks$. 

For the higher order continuum model, assume that the external force $f \in \Us^{1,2}$, we seek the solution for the following variational problem:
\begin{equation}
u^{\hoc} \in \argmin \Big\{ \Es_\hoc(u) - \<f, u\>_{\Omega} \big| u \in \Us^{1, 2} \cap \Ks \Big \}.
\label{HOCprob}
\end{equation}
The solution $u^{\hoc}$ to \eqref{HOCprob} satisfies the first-order condition 
\begin{equation}
\<\delta \Es_\hoc(u^{\hoc}), \^v\> = \<f, \^v\>_{\Omega}, \quad \forall \^v \in \Us^{1,2}.
\label{HOCvarprob}
\end{equation}
We call the solution $u^{\hoc}$ of \eqref{HOCprob} (strongly) stable, if there exists a positive number $\gamma_0$ such that
\begin{equation}\label{hocstab}
\<\delta^2 \Es_\hoc(u^{\hoc})\^v, \^v\> \geq \gamma_0||\nabla \^v||^2_{L^2(\Omega)}, \quad \forall \^v \in \Us^{1,2}.
\end{equation}

\def\uhoc{u^{\hoc}}

\section{Modeling Error Analysis}
\label{sec:modelingerror}
In this section, we give a rigorous analysis of the modeling error. We first introduce the atomistic stress tensor $S^{\a}(u;x)$ and the stress of the higher continuum model $S^{\hoc}(u;x)$ in Section \ref{sec:A&C stresses}. Then we derive the pointwise error estimate in stress $R(u;x) = S^{\a}(u;x) - S^{\hoc}(u;x)$ in Section \ref{sec:postress}. Finally, we present the fourth-order consistency estimate of the higher order continuum model \eqref{hocmodel} in Section \ref{Sec::model}. 

\subsection{Atomistic and continuum stresses}
\label{sec:A&C stresses}
The first variation of the atomistic energy functional $\Es_\a$ \eqref{Aenergyinu} at $u \in \Uc^{1,2}$, is given by
\begin{equation}
\<\delta\Es_{\a}(u),v\>=\sum_{\xi\in\ol}\sum_{\rho \in \Rc}\phi'_{\rho}(D_{\rho}u(\xi))\cdot D_{\rho}v(\xi),   \qquad \forall v\in \Uc^{1,2}. 
\end{equation}
We replace the test function $v$ by its convolution based quasi-interpolation $\~v$ and apply the localization formula \eqref{eq:localization3}, it follows that
\begin{align}
\<\delta\Es_{\a}(u),\~v\>&=\sum_{\xi\in\ol}\sum_{\rho \in \Rc}\phi'_{\rho}(D_{\rho}u(\xi))\cdot\int_{\R} \chi_{\xi,\rho}(x) \nabla_{\rho}\^v(x)dx \nonumber \\ 
&=\int_{\R}\Big[\sum_{\xi\in\ol}\sum_{\rho \in \Rc}(\rho\phi'_{\rho}(D_{\rho}u(\xi)))\chi_{\xi,\rho}(x)\Big]\cdot\nabla \^v dx \nonumber \\ 
&=\int_{\Omega}\Big[\sum_{\xi\in\Lambda}\sum_{\rho \in \Rc}(\rho\phi'_{\rho}(D_{\rho}u(\xi)))\chi_{\xi,\rho}(x)\Big]\cdot\nabla \^v dx \nonumber \\ 
&=: \int_{\Omega}S^{\a}(u;x)\cdot\nabla \^v(x)dx,
\label{StressAdef}
\end{align}
where 
\begin{equation}
S^{\a}(u;x) := \sum_{\xi\in\Lambda}\sum_{\rho \in \Rc}(\rho\phi'_{\rho}(D_{\rho}u(\xi)))\chi_{\xi,\rho}(x),
\label{eq:StressA}
\end{equation}
is defined to be the atomistic stress tensor. The last two identities in \eqref{StressAdef} hold because of the periodic boundary condition and $\chi_{\xi,\rho}(x) = 0$ when $|\xi - x| > 2r_{\cut}$.

The first variation of the higher order continuum energy functional $\Es_{\hoc}$ defined in \eqref{hocmodel} is given by
\begin{equation}
\<\delta\Es_{\hoc}(u),\^v\>=\int_{\Omega}\sum_{\rho \in \Rc}\phi'_{\rho}(\rho \nabla u+\frac{\rho^3}{24}\nabla^3u)(\rho \nabla \^v+\frac{\rho^3}{24}\nabla^3 \^v)dx, \qquad \forall \^v \in \Us^{1, 2}.
\label{eq:HOCfv}
\end{equation}
Integration by parts and the periodic boundary condition of the test function lead to
\begin{align}
\label{firstvariation}
\<\delta\Es_{\hoc}(u),\^v\>
:=& \int_{\Omega}S^{\hoc}(u;x)\cdot\nabla \^v(x)dx, \qquad \forall \^v \in \Us^{1,2},
\end{align}
where 
\begin{align} \label{eq:StressHOC}
S^{\hoc}(u; x) :=& \sum_{\rho \in \Rc} \Big[ \rho\phi'_{\rho}(\rho\nabla u+\frac{\rho^3}{24}\nabla^3 u)+\frac{\rho^3}{24}\phi'''_{\rho}(\rho\nabla u+\frac{\rho^3}{24}\nabla^3 u)(\rho\nabla^2u+\frac{\rho^3}{24}\nabla^4 u)^2 \nonumber \\ 
 &+\frac{\rho^4}{24}\phi''_{\rho}(\rho\nabla u+\frac{\rho^3}{24}\nabla^3 u)\nabla^3 u+\frac{\rho^6}{576}\phi''_{\rho}(\rho\nabla u+\frac{\rho^3}{24}\nabla^3 u)\nabla^5 u \Big],
 \end{align}
is defined to be the stress of the higher order continuum model \eqref{hocmodel}. 

By the integration by parts to \eqref{firstvariation} again, we obtain
\begin{align}
\<\delta\Es_{\hoc}(u),\^v\>
:= & \int_{\Omega} W^{\hoc}(u) \cdot \^v(x)dx, \qquad \forall \^v \in \Us^{1,2},
\end{align}
where 
\begin{align} 
\label{hocEL}
W^{\hoc}(u) :=& \sum_{\rho \in \Rc} \Big[ \rho\phi'_{\rho}(\rho\nabla u+\frac{\rho^3}{24}\nabla^3 u)(\rho\nabla^2 u+\frac{\rho^3}{24}\nabla^4 u) \nonumber \\ 
&+\frac{\rho^3}{24}\phi'''_{\rho}(\rho\nabla u+\frac{\rho^3}{24}\nabla^3 u)(\rho\nabla^2u+\frac{\rho^3}{24}\nabla^4 u)^3 \nonumber \\ 
 &+\frac{\rho^3}{12}\phi'''_{\rho}(\rho\nabla u+\frac{\rho^3}{24}\nabla^3 u)(\rho\nabla^2u+\frac{\rho^3}{24}\nabla^4 u)(\rho\nabla^3u+\frac{\rho^3}{24}\nabla^5 u) \nonumber \\
 &+ \frac{\rho^3}{24}\phi'''_{\rho}(\rho\nabla u+\frac{\rho^3}{24}\nabla^3 u)(\rho\nabla^2u+\frac{\rho^3}{24}\nabla^4 u)\nabla^3 u \nonumber \\ 
 &+\frac{\rho^6}{576}\phi''_{\rho}(\rho\nabla u+\frac{\rho^3}{24}\nabla^3 u)(\rho\nabla^2 u+\frac{\rho^3}{24}\nabla^4 u)\nabla^5 u \nonumber \\ 
 &+\frac{\rho^4}{24}\phi''_{\rho}(\rho\nabla u+\frac{\rho^3}{24}\nabla^3 u)\nabla^4 u +  \frac{\rho^6}{576}\phi''_{\rho}(\rho\nabla u+\frac{\rho^3}{24}\nabla^3 u)\nabla^6 u \Big].
 \end{align}
$W^{\hoc}(u) = 0$ is the Euler-Lagrange equation of the higher order continuum model, which is a sixth order nonlinear elliptic equation. 

We now define the error in stress as $R(u;x) := S^{\a}(u;x) - S^{\hoc}(u;x)$. In the remaining part of this section, we will give the pointwise estimate of $R(u;x)$ and show the fourth order consistency of the higher order continuum model \eqref{hocmodel}.

\subsection{Pointwise estimate of the error in stress}
\label{sec:postress}
In this section we prove the pointwise estimate of $R(u;x)$, the error in stress. We first introduce a useful lemma which is a direct extension of \cite[Lemma 11]{OrtnerTheil2012}.

\begin{lemma}
\label{lemma1}
Let $x,\rho \in \Omega$, $k=0,1,2,3$, and $\chi_{\xi,\rho}(x)$ is defined by \eqref{eq:localization3}. We have
\begin{equation}
\sum_{\xi \in \Lambda}\chi_{\xi,\rho}(x)(\xi-x)^k=\frac{(-\rho)^k}{k+1}.
\end{equation}
\end{lemma}
\begin{proof}
This result relies on the assumption that it is true on a shifted grid: if $v:\R \rightarrow \R$ is a polynomial whose order is less than $k$, where $k=0,1,2,3$, then for any $z, x \in \R$ we have
\begin{equation}
v(x)=\sum_{\eta\in(\Lambda+z)}\zeta(x-\eta)v(\eta).
\label{lemma_interp}
\end{equation}
To prove the result, let $s\in[0,1]$ be fixed, then
$$\sum_{\xi\in\Lambda}\zeta((\xi-x)+s\rho)(\xi-x)^k=\sum_{\eta\in(x+\Lambda)}\zeta(s\rho-\eta)(-\eta)^k,$$
where we substituted $\eta=-(\xi-x)$ and employing \eqref{lemma_interp} with $v(x)\equiv(-x)^k$, we obtain
$$\sum_{\xi\in\Lambda}\zeta((\xi-x)+s\rho)(\xi-x)^k=(-s\rho)^k.$$
By the definition of $\chi_{\xi,\rho}(x)$ in \eqref{eq:localization3} and by integrating w.r.t. $s$, we have
\begin{align}
\sum_{\xi\in\Lambda}\chi_{\xi,\rho}(x)(\xi-x)^k&=\int_0^1\sum_{\xi\in\Lambda}\zeta((\xi-x)+s\rho)(\xi-x)^kds \nonumber \\ 
&=\int_0^1(-s\rho)^kds=\frac{(-\rho)^k}{k+1}.
\end{align} 
It is trivial to see that $\sum_{\xi \in \Lambda}\chi_{\xi,\rho}(x)=1$ if we let $k=0$.
\end{proof}

The pointwise estimate of the error in stress is given by the following lemma.

\begin{lemma}
\label{Thm::1}
Let $u\in W^{5, \infty} \cap \Ks$, and $x\in\Omega$, then
\begin{align}
\big|R(u;x)\big| \leq & C(||\nabla^5u||_{L^{\infty}(v_{x})} + ||\nabla^2u\nabla^4u||_{L^{\infty}(v_{x})}  \nonumber \\
&+||\nabla^3u(\nabla^2u)^2||_{L^{\infty}(v_{x})} +||\nabla^3u||^2_{L^{\infty}(v_{x})}+||\nabla^2u||^4_{L^{\infty}(v_{x})}),
\end{align}
where $C$ depends on $M^{(j, 4)}, j = 2,...,5$, defined in Section \ref{sec:formulation:atm}, and $v_x:=B_{2r_{cut}+1}(x)$ is the neighbourhood of some $x\in\R$ and $r_{cut} = \max_{r\in \Rc} |r|$. 
\end{lemma}

\begin{proof}
In order to keep the notation concise, we first define
$$\tau_{j} := ||\nabla^j u||_{L^{\infty}(v_x)}, j = 2,3,4,5.$$

By a direct Taylor expansion of \eqref{eq:StressHOC} and using the fact $\sum_{\xi \in \Lambda}\chi_{\xi,\rho}(x)=1$, we can rewrite the stress of the higher order continuum model as
\begin{align}
S^{\hoc}(u;x) = &\sum_{\xi\in\Lambda}\sum_{\rho \in \Rc} \Big[ \rho\phi'_{\rho}(\nabla_{\rho} u)+\frac{\rho^4}{12}\phi''_{\rho}(\nabla_{\rho} u)\nabla^3u+\frac{\rho^5}{24}\phi'''_{\rho}(\nabla_{\rho} u)(\nabla^2u)^2  \nonumber \\
%%%%%%%%%%%%%%
&+\frac{\rho^6}{576}\phi''_{\rho}(\mu_1)\nabla^5u+\frac{\rho^7}{384}\phi'''_{\rho}(\mu_2)(\nabla^3u)^2+\frac{\rho^7}{288}\phi'''_{\rho}(\mu_3)\nabla^2u\nabla^4u \nonumber  \\ 
%%%%%%%%%%%%%
&+\frac{\rho^8}{576}\phi^{(4)}_{\rho}(\mu_4)\nabla^3u(\nabla^2u)^2\Big] \chi_{\xi,\rho}(x), 
\label{2}
\end{align}
where $\mu_j \in \conv\{\nabla_{\rho}u, \nabla_{\rho}u+\frac{\rho^3}{24}\nabla^3u\}, j=1,2,3,4$.
%Hence, $R(u;x)$ can be written as 
%\begin{equation}\label{Stressdiff}
%R(u;x)=\sum_{\xi\in\Lambda}\sum_{\rho \in \Rc}\left( \rho\cdot T\right)\chi_{\xi,\rho}(x),
%\end{equation}
%where
%\begin{align}
% T = & \phi'_{\rho}(D_{\rho}u(\xi)) -\phi'_{\rho}(\rho\nabla u)-\frac{\rho^3}{12}\phi''_{\rho}(\rho\nabla u)\nabla^3u-\frac{\rho^4}{24}\phi'''_{\rho}(\rho\nabla u)(\nabla^2u)^2 \\ \nonumber
%&-\frac{\rho^5}{576}\phi''_{\rho}(\rho\nabla u)\nabla^5u-\frac{\rho^6}{384}\phi'''_{\rho}(\rho\nabla u)(\nabla^3u)^2-\frac{\rho^6}{288}\phi'''_{\rho}(\rho\nabla u)\nabla^2u\nabla^4u \\ \nonumber
%&-\frac{\rho^7}{576}\phi^{(4)}_{\rho}(\rho\nabla u)\nabla^3u(\nabla^2u)^2.
%\end{align}

We then turn our attention to the atomistic stress tensor in \eqref{eq:StressA} where 
\begin{equation}\label{1}
S^{\a}(u;x)=\sum_{\xi\in\Lambda}\sum_{\rho \in \Rc}(\rho\phi'_{\rho}(D_{\rho}u(\xi)))\chi_{\xi,\rho}(x).
\end{equation}
Since $\zeta$ has a compact support, we have $\chi_{\xi,\rho}(x) = 0$ for all $\xi \in \Lambda$ with $|\xi - x| > 2|\rho|$. We thus can apply Taylor expansion to the term $\phi'_{\rho}(D_{\rho}u(\xi))$ at $x$. % in $v_x:=B_{2r_{\cut}+1}(x)$. in the neighbourhood of $\xi$ \commentlz{check}. 
We begin by expanding $D_{\rho}u(\xi)$ for $\rho \in \Rc$ in $v_x$, in the neighbourhood of $x$, so that
\begin{align}
D_{\rho}u(\xi) =& \nabla_{\rho}u + \Big[ \rho(\xi-x) + \frac{\rho^2}{2}\Big]\nabla^2u + \Big[ \frac{\rho}{2}(\xi-x)^2 + \frac{\rho^2}{2}(\xi-x) + \frac{\rho^3}{6} \Big]\nabla^3u  \nonumber \\
&+ \Big[ \frac{\rho}{6}(\xi-x)^3 + \frac{\rho^2}{4}(\xi-x)^2 + \frac{\rho^3}{6}(\xi-x)+\frac{\rho^4}{24}\Big]\nabla^4u + O(\tau_5).\label{3}
\end{align}
The fact that $D_{\rho}u(\xi) - \nabla_{\rho}u = O(\tau_2)$ allows us to expand $\phi'_{\rho}(D_{\rho}u(\xi))$ as
\begin{align}\label{4}
\phi'_{\rho}(D_{\rho}u(\xi)) =& \phi'_{\rho}(\nabla_{\rho}u) + \phi''_{\rho}(\nabla_{\rho}u)\big( D_{\rho}u(\xi) - \nabla_{\rho}u \big) + \1/2\phi'''_{\rho}(\nabla_{\rho}u)\big( D_{\rho}u(\xi) - \nabla_{\rho}u \big)^2  \nonumber \\
&+ \frac{1}{6}\phi^{(4)}_{\rho}(\nabla_{\rho}u)\big( D_{\rho}u(\xi) - \nabla_{\rho}u \big)^3 + \frac{1}{24}\phi^{(5)}_{\rho}(\mu_5)\big( D_{\rho}u(\xi) - \nabla_{\rho}u \big)^4,
\end{align}
where $\mu_5 \in \conv\{\nabla_{\rho}u, D_{\rho}u(\xi)\}$.

Combining \eqref{1}, \eqref{2}, \eqref{3} and \eqref{4} and after some algebraic manipulation, we have
\begin{align}
 R(u;x) = \sum_{\xi\in\Lambda}\sum_{\rho \in \Rc} &\Big\{ \rho\phi''_{\rho}(\nabla_{\rho} u)\Big[\rho(\xi-x)+\frac{\rho^2}{2}\Big]\nabla^2u \tag{a}\\ 
&+ \rho \phi''_{\rho}(\nabla_{\rho} u)\Big[\frac{\rho}{2}(\xi-x)^2+\frac{\rho^2}{2}(\xi-x)+\frac{\rho^3}{12}\Big]\nabla^3u \tag{b} \\ 
%%%%%%%%%%%%%%%%
&+ \rho \phi'''_{\rho}(\nabla_{\rho} u)\Big[ \frac{\rho^2}{2}(\xi-x)^2 + \frac{\rho^3}{2}(\xi-x) + \frac{\rho^4}{6}\Big](\nabla^2 u)^2 \nonumber \tag{c}\\ 
%%%%%%%%%%%%%%%%%%%%%%%%%%%%
&+ \rho \phi''_{\rho}(\nabla_{\rho} u)\Big[ \frac{\rho}{6}(\xi-x)^3+\frac{\rho^2}{4}(\xi-x)^2+\frac{\rho^3}{6}(\xi-x)+\frac{\rho^4}{24}\Big]\nabla^4u \tag{d}\\ 
%%%%%%%%%%%%%%%%%%%%%%%%%%%%
&+ \rho \phi'''_{\rho}(\nabla_{\rho} u)\Big[ \frac{\rho^2}{2}(\xi-x)^3+\frac{3\rho^3}{4}(\xi-x)^2+\frac{5\rho^4}{12}(\xi-x)+\frac{\rho^5}{12}\Big]\nabla^2u\nabla^3u \tag{e} \\ 
%%%%%%%%%%%%%%%%%%%%%%%%%%%%
&+ \rho \phi^{(4)}_{\rho}(\nabla_{\rho} u)\Big[ \frac{\rho^3}{6}(\xi-x)^3 + \frac{\rho^4}{4}(\xi-x)^2 + \frac{\rho^5}{8}(\xi-x)+\frac{\rho^6}{48}\Big](\nabla^2u)^3 \tag{f} \\ 
&+ O(\tau_5) + O(\tau_2\tau_4) + O(\tau^2_3) + O(\tau_3\tau^2_2) + O(\tau^4_2)
%%%%%%%%%%%%%%%%%%%%%%%%%%%%
%&+\phi''_{\rho}(\rho\nabla u)\Big[ \frac{\rho}{24}(\xi-x)^4 + \frac{\rho^2}{12}(\xi-x)^3 +\frac{\rho^3}{12}(\xi-x)^2 +\frac{\rho^4}{24}(\xi-x)+ \frac{19\rho^5}{2880}\Big]\nabla^5u \\ \nonumber
%%%%%%%%%%%%%%%%%%%%%%%%%%%%%
%&+\frac{1}{2}\phi'''_{\rho}(\rho\nabla u)\Big[ \frac{\rho}{2}(\xi-x)^2+\frac{\rho^2}{2}(\xi-x)+\frac{\rho^3}{6}\Big]^2(\nabla^3u)^2 \\ \nonumber
%%%%%%%%%%%%%%%%%%%%%%%%%%%%%
%&+\frac{1}{2}\phi'''_{\rho}(\rho\nabla u)\Big[ \rho(\xi-x)+\frac{\rho^2}{2}\Big]\Big[ \frac{\rho}{6}(\xi-x)^3+\frac{\rho^2}{4}(\xi-x)^2+\frac{\rho^3}{6}(\xi-x)+\frac{\rho^4}{24}\Big]\nabla^2u\nabla^4u \\ \nonumber
%%%%%%%%%%%%%%%%%%%%%%%%%%%%%
%&+\frac{1}{6}\phi^{(4)}_{\rho}(\rho\nabla u)\Big[ \rho(\xi-x)+\frac{\rho^2}{2}\Big]^2\Big[ \frac{\rho}{2}(\xi-x)^2+\frac{\rho^2}{2}(\xi-x)+\frac{\rho^3}{6}\Big]\nabla^3u(\nabla^2u)^2 \\ \nonumber
%%%%%%%%%%%%%%%%%%%%%%%%%%%%%
%&+\frac{1}{24}\phi^{(5)}_{\rho}(\rho\nabla u)\Big[ \rho(\xi-x)+\frac{\rho^2}{2}\Big]^4(\nabla^2u)^4 
\Big\} \chi_{\xi, \rho}(x).
\label{Tylorexp}
\end{align}
%Inserting the Taylor expansion \eqref{Tylorexp} into \eqref{Stressdiff}, rearranging the sum in \eqref{Stressdiff} and 
Lemma \ref{lemma1} leads to the following identities, which result in the elimination of $(a)-(f)$ terms in \eqref{Tylorexp}. 
\begin{align}
&\Big( \sum_{\xi\in\Lambda}\chi_{\xi,\rho}(x)\Big)\Big[ \rho(\xi-x)+\frac{\rho^2}{2}\Big]=0, \nonumber \\ 
%%%%%%%%%%%%%%%%%%%%%%%%%%%%
&\Big( \sum_{\xi\in\Lambda}\chi_{\xi,\rho}(x)\Big)\Big[\frac{\rho}{2}(\xi-x)^2+\frac{\rho^2}{2}(\xi-x)+\frac{\rho^3}{12}\Big]=0, \nonumber \\ 
%%%%%%%%%%%%%%%%%%%%%%%%%%%%
&\Big( \sum_{\xi\in\Lambda}\chi_{\xi,\rho}(x)\Big)\Big[ \frac{\rho^2}{2}(\xi-x)^2 + \frac{\rho^3}{2}(\xi-x) + \frac{\rho^4}{6}\Big]=0, \nonumber \\ 
%%%%%%%%%%%%%%%%%%%%%%%%%%%%
&\Big( \sum_{\xi\in\Lambda}\chi_{\xi,\rho}(x)\Big)\Big[ \frac{\rho}{6}(\xi-x)^3+\frac{\rho^2}{4}(\xi-x)^2+\frac{\rho^3}{6}(\xi-x)+\frac{\rho^4}{24}\Big]=0 ,\nonumber \\ 
%%%%%%%%%%%%%%%%%%%%%%%%%%%%
&\Big( \sum_{\xi\in\Lambda}\chi_{\xi,\rho}(x)\Big)\Big[ \frac{\rho^2}{2}(\xi-x)^3+\frac{3\rho^3}{4}(\xi-x)^2+\frac{5\rho^4}{12}(\xi-x)+\frac{\rho^5}{12}\Big]=0 ,\nonumber \\ 
%%%%%%%%%%%%%%%%%%%%%%%%%%%%
&\Big( \sum_{\xi\in\Lambda}\chi_{\xi,\rho}(x)\Big)\Big[ \frac{\rho^3}{6}(\xi-x)^3 + \frac{\rho^4}{4}(\xi-x)^2 + \frac{\rho^5}{8}(\xi-x)+\frac{\rho^6}{48}\Big]=0.
\label{iden1}
\end{align}

We now only need to estimate the remaining terms in \eqref{Tylorexp}. By the definition of $\chi_{\xi,\rho}$, we have the estimate $\sum_{\xi\in\Lambda}\chi_{\xi,\rho}h_{\xi} \leq \max_{\xi\in\Lambda , \chi_{\xi,\rho}\neq0 }h_{\xi}$, where $h_{\xi}$ is an arbitrary function with respect to $\xi$. 
Combined with the boundedness of the derivatives of the interaction potential $\phi_{\rho}$ assumed in \eqref{eq: boundedness of the potential},  it is easy to show that
\begin{align*} 
|R(u;x)| \lesssim& M^{(2,4)}||\nabla^5u||_{L^{\infty}(v_{x})} + M^{(3,4)}||\nabla^2u\nabla^4u||_{L^{\infty}(v_{x})} \\ \nonumber
&+M^{(4,4)}||\nabla^3u(\nabla^2u)^2||_{L^{\infty}(v_{x})} +M^{(3,4)}||\nabla^3u||^2_{L^{\infty}(v_{x})}+M^{(5,4)}||\nabla^2u||^4_{L^{\infty}(v_{x})},
%\sum_{\xi\in\Lambda}\sum_{\rho \in \Rc}\left( \rho\cdot T_1\right)\chi_{\xi,\rho}(x) & \lesssim M_s^{(2)}||\nabla^5u||_{L^{\infty}(v_{x})}, \\
%%%%%%%%
%\sum_{\xi\in\Lambda}\sum_{\rho \in \Rc}\left( \rho\cdot T_2\right)\chi_{\xi,\rho}(x) &\lesssim M^{(3)}_s||\nabla^3u||^2_{L^{\infty}(v_{x})}, \\
%%%%%%%%
%\sum_{\xi\in\Lambda}\sum_{\rho \in \Rc}\left( \rho\cdot T_3\right)\chi_{\xi,\rho}(x) &\lesssim M^{(3)}_s||\nabla^2u\nabla^4u||_{L^{\infty}(v_{x})}, \\
%%%%%%%%
%\sum_{\xi\in\Lambda}\sum_{\rho \in \Rc}\left( \rho\cdot T_4\right)\chi_{\xi,\rho}(x) &\lesssim M^{(4)}_s||\nabla^3u(\nabla^2u)^2||_{L^{\infty}(v_{x})}, \\
%%%%%%%%
%\sum_{\xi\in\Lambda}\sum_{\rho \in \Rc}\left( \rho\cdot T_5\right)\chi_{\xi,\rho}(x) &\lesssim M^{(5)}_s||\nabla^2u||^4_{L^{\infty}(v_{x})},
\end{align*}
and this finishes the proof.
\end{proof}

\begin{remark}\label{re}
Notice that the lower order terms ($\leq 4$) in \eqref{Tylorexp} vanish for different reasons. Terms $(a)$, $(d)$, $(e)$, $(f)$ come from the atomistic model only. Terms $(b)$, $(c)$ come from both atomistic model and higher order continuum model, and the cancellation of those terms depends on the choice of expansion points in \eqref{eq: Taylor expansion of u}. For example, if we choose the lattice points instead of the mid points as expansion points in \eqref{eq: Taylor expansion of u}, we can get the following higher order continuum energy,
\begin{equation}
\Es^{\rm fir}_\hoc(u)=\int_{\Omega}\sum_{\rho\in \Rc} \phi_{\rho}(\rho\nabla u+\frac{\rho^2}{2}\nabla^2u + \frac{\rho^3}{6}\nabla^3u)dx,
\label{hocmodel can not}
\end{equation}
which is only first order accurate. 
\end{remark}

\begin{remark}

In fact, the higher order continuum model which is originally derived in \cite{Triantafyllidis1993},
\begin{equation}
\Es^{\rm fou}_{\hoc}(u)=\int_{\Omega}\sum_{\rho\in\Rc}\big[\phi_{\rho}(\rho \nabla u)-\frac{\rho^4}{24}\phi''_{\rho}(\rho\nabla u)(\nabla^2u)^2\big]dx, 
\label{hocmodel can}
\end{equation}
also has fourth order estimate in stress. As a matter of fact, \eqref{hocmodel} and \eqref{hocmodel can} differs only a null-Lagrangian for second order term. However, this model is ill-posed and thus is not stable. See Remark \ref{posedness} for more details.

We give a simplified error analysis for the higher order continuum model with 6th order accuracy in Appendix A, which truncate the terms of order 5 onwards in \eqref{eq: Taylor expansion of u}. 

From those observations, we conjecture that: we need to include higher order gradient up to $2k+1$th order to obtain a "well-posed" higher order continuum model of order $2k+2$.

% the even order gradients should not be included in the continuum model to achieve the best order of accuracy \ys{and satisfy the well-posedness.} \sout{and} \ys{We should apply the inner expansion technique with the expansion points being the mid points of the interaction bonds, and truncate the terms of order $2k+1$ onwards to obtain a continuum model with $2k+2$'th order of accuracy.} However, rigorous proof of such conjecture is not trivial. The first step is the extension of Lemma \ref{lemma1} to arbitrary $k$ which depends on the property of the nodal basis function $\zeta$ that it preserves $2k+1$'th order polynomials (c.f. \eqref{eq: zeta preserving 3rd order polynomials}). Such property can be achieved by using higher order splines basis function. Further steps towards the rigorous proof become involved since integration by parts and Taylor expansion are repeatedly applied before we obtain the equation of stress difference similar to \eqref{Tylorexp} where the cancellations in the coefficients of the lower order terms begin to appear. We believe some technicalities similar to those for the construction of Butcher's table might be created to simplify the procedure of constructing arbitrary higher order continuum model. 

%We also note that, in addition to the \ys{possible} decrease of accuracy, the inclusion of the strain gradient $\nabla^2u$ in the continuum model may cause the issue of ill-posedness. We refer to \ys{Remark \ref{posedness} and} \cite[]{arndt2005derivation} for a detailed discussion.\ys{I think this paragraph should be put into Remark 4.1...}

\end{remark}

Construction of higher order continuum model for more physical relevant cases of multi body iterations and/or higher dimensions will be discussed in Section \ref{sec:conclusion}. In those cases, a similar but more involved formulation of stress differences as \eqref{Tylorexp} will serve as the key of developing and analyzing higher order continuum models.  

\subsection{Fourth-order consistency of the higher order continuum model}
\label{Sec::model}
Lemma \ref{Thm::1} gives us the upper bound of $|R(u;x)|$, we now convert this pointwise estimate into a global estimate. The main idea is to use the inverse estimates to obtain $L^2$ type bounds from the $L^{\infty}$ bounds. It is easy to show that
\begin{equation}
\label{inverse}
||\nabla^j \Pi v||_{L^{\infty}(T)} \lesssim ||\nabla^j \Pi v||_{L^2(T)}, \forall v \in \Uc^{1,2}, \text{where $T$ is any bounded domain}, j = 0,1,..., 5,
\end{equation}
where the interpolation operator $\Pi: \Uc^{1,2} \rightarrow C^4$ is defined in \eqref{eq: 5th order finite difference}.

\begin{theorem}\label{thm2}
Let $u\in\Uc^{1,2} \cap \Kc$ and $\~v = \zeta * \^v$. The interpolation operator $\Pi: \Uc^{1,2} \rightarrow C^4$ is defined in Section \ref{sec: interpolation based on finite differences} by  \eqref{eq: 5th order finite difference}. Then the model error is bounded by
\begin{align}
\<\delta\Es_{\hoc}&(\Pi u), \^v\> - \<\delta \Es_{\a} (u), \~v\> \leq C(||\nabla^5 \Pi u||_{L^2(\Omega)} + ||\nabla^2\Pi u\nabla^4\Pi u||_{L^2(\Omega)} \nonumber \\ 
+& ||\nabla^3\Pi u(\nabla^2\Pi u)^2||_{L^2(\Omega)} + ||\nabla^3\Pi u||^2_{L^4(\Omega)} + ||\nabla^2\Pi u||^4_{L^8(\Omega)})||\nabla \^v||_{L^2(\Omega)}, \quad \forall \^v \in \Us^{1,2},
\end{align}
where $C$ depends on $M^{(j, 4)}, j=2,...,5$ which are defined in Section \ref{sec:formulation:atm}.
\end{theorem}
\begin{proof}
From the definition of $\Pi$ by \eqref{eq: 5th order finite difference}, we observe that $\Pi u |_{\ol} = u | _{\ol}$. It follows from the localization formula \eqref{eq:localization3} that, 
\begin{align}
&\<\delta\Es_{\hoc}(\Pi u), \^v\> - \<\delta \Es_{\a} (u), \~v\> \nonumber  \\
=& \<\delta\Es_{\hoc}(\Pi u), \^v\> - \<\delta \Es_{\a} (\Pi u), \~v\> \nonumber \\ 
=&  \int_{\Omega} R(\Pi u;x)\cdot \nabla \^v dx.
\end{align}
An application of Cauchy-Schwarz inequality yields
$$\int_{\Omega} R(\Pi u;x)\cdot \nabla \^v dx \leq ||R(\Pi u;x)||_{L^2(\Omega)}\cdot||\nabla \^v||_{L^2(\Omega)}.$$
Using the inverse estimates \eqref{inverse}, we obtain
\begin{align}
|R(\Pi u; x)|^2  \lesssim & ||\nabla^5 \Pi u||^2_{L^{\infty}(v_x)} + ||\nabla^2\Pi u\nabla^4\Pi u||^2_{L^{\infty}(v_x)} \nonumber \\ 
&+ ||\nabla^3\Pi u (\nabla^2\Pi u)^2||^2_{L^{\infty}(v_x)} + ||\nabla^3\Pi u||^4_{L^{\infty}(v_x)} + ||\nabla^2\Pi u||^8_{L^{\infty}(v_x)} \nonumber \\ 
\lesssim & ||\nabla^5 \Pi u||^2_{L^2(v_x)} + ||\nabla^2\Pi u\nabla^4\Pi u||^2_{L^2(v_x)}  \nonumber \\
&+ ||\nabla^3\Pi u (\nabla^2\Pi u)^2||^2_{L^2(v_x)} + ||\nabla^3\Pi u||^4_{L^4(v_x)} + ||\nabla^2\Pi u||^8_{L^8(v_x)},
\label{R}
\end{align}
where $v_x$ is a compact support of $x$ defined in Lemma \ref{Thm::1}. Integrating \eqref{R} over $\Omega$, we have
\begin{align}
||R(\Pi u;x)||_{L^2(\Omega)} \lesssim & ||\nabla^5 \Pi u||_{L^2(\Omega)} + ||\nabla^2\Pi u\nabla^4\Pi u||_{L^2(\Omega)}  \nonumber \\
&+ ||\nabla^3\Pi u (\nabla^2\Pi u)^2||_{L^2(\Omega)} + ||\nabla^3\Pi u||^2_{L^4(\Omega)} + ||\nabla^2 \Pi u||^4_{L^8(\Omega)}, 
\end{align}
which yields the stated result.
\end{proof}

\section{Stability}
\label{sec:stability}
In this section, we present the stability estimate of the higher order continuum model \eqref{hocmodel}, which extends the stability results for Cauchy-Born model in \cite[Theorem 3.1]{Hudson:stab}. 
%For the sake of convenience we will now directly use the notation of that paper. 
%Let $y_{\rm F}=Fx$ be the homogeneous deformation given $F > 0$, which is a macroscopic deformation gradient. 
The atomistic model \eqref{Aenergyiny} and its solution space $\Uc^{1, 2}$ can be denoted as $\Es_{\a}^{N}(u)$ and $\Uc^{1,2}_N$ since they actually depend on the computational domain $\ol = [-N, N] \cap \L$. For a fixed $N \in \N$, given potential $\phi_{\rho}(r)=\phi(r+F\rho)$ defined in Section \ref{sec:formulation:atm}, we call the homogeneous deformation $y=Fx$ is stable in the finite atomistic model if 
 \begin{align}
 \L^{N}_{\a} := \inf_{\substack{v\in\Uc_N^{1,2},\\ ||\nabla v||_{L^2(\Omega)}=1}}\<\delta^2\Es^N_{\a}(0)v,v\>>0.
\label{StabConst:An}
\end{align}
We require a stronger definition of the stability in the infinite atomistic model: 
 \begin{align}
 \L_{\a} := \inf_{N \in \N}  \L^{N}_{\a} >0.
\label{StabConst:A}
\end{align}
Also, the homogeneous deformation is stable for the Cauchy-Born model \eqref{cbmodel} if
 \begin{align}
 \L_{\cb} := \inf_{\substack{\^v\in W^{1,2},\\ ||\nabla \^v||_{L^2(\Omega)}=1}}\<\delta^2\Es_{\cb}(0)\^v, \^v\>>0,
\label{StabConst:C}
\end{align}
and the homogeneous deformation is stable for the higher order continuum model \eqref{hocmodel} if
 \begin{align}
 \L_{\hoc} := \inf_{\substack{\^v\in \Us^{1,2},\\ ||\nabla \^v||_{L^2(\Omega)}=1}}\<\delta^2\Es_{\hoc}(0)\^v, \^v\>>0.
\label{StabConst:HOC}
\end{align}

The following lemma states the stability of the higher order continuum model \eqref{hocmodel} at the homogeneous deformation, namely: the stability of the atomistic model implies that of the higher order continuum model, and the stability of the higher order continuum model \eqref{hocmodel} is "in between" the atomistic model and the Cauchy-Born model.

\begin{lemma}
\label{lmm: stability}
If the deformation gradient $F$ introduced in Section \ref{sec:formulation:atm} is positive, then $\L_{\a}\leq \L_{\hoc} \leq \L_{\cb}$.
\end{lemma}
\begin{proof}
For the first inequality, we extend the proof given in \cite[Section 3.1]{Hudson:stab} to higher order continuum model. The energy $\Es^N_{\a}$ and $\Es_{\hoc}$ can be expanded up to second order for an arbitrary small $t>0$ and $\^v \in C^3(\Omega)$, 
\begin{align}
\Es^N_{\a}(0+t\^v) &= \Es^N_{\a}(0) + \frac{t^2}{2}\<\delta^2\Es^N_{\a}(0)\^v, \^v\> + r_N,  \nonumber \\ 
\Es_{\hoc}(0+t\^v) &= \Es_{\hoc}(0) + \frac{t^2}{2}\<\delta^2\Es_{\hoc}(0)\^v, \^v\> + r_{\hoc}, 
\end{align}
where 
$$ |r_N| + |r_{\hoc}| \lesssim t^3 ||\nabla^3 \^v||^3_{L^{\infty}(\Omega)}.$$
By \eqref{hocmodel} and \eqref{Aenergyinu}, we have $\Es^N_{\a}(0) = \Es_{\hoc}(0)$ and $\lim_{N \rightarrow \infty} \Es^N_{\a}(0+t\^v) = \Es_{\hoc}(0+t\^v)$.
Hence we have that
$$\limsup_{N \rightarrow \infty} |\<\delta^2\Es^N_{\a}(0)\^v, \^v\> - \<\delta^2\Es_{\hoc}(0)\^v, \^v\>| \leq \limsup_{N \rightarrow \infty} \frac{2}{t^2}|r_N - r_{\hoc}| \lesssim t||\nabla^3 u||^3_{L^{\infty}(\Omega)},$$ 
letting $t \rightarrow 0$, we have $\lim_{N \rightarrow \infty} \<\delta^2\Es^N_{\a}(0)\^v, \^v\> = \<\delta^2\Es_{\hoc}(0)\^v, \^v\>$. According to the definition of $\L_{\hoc}$, there exists $\^v_{\delta} \in \Us^{1,2}$ such that $||\nabla \^v_{\delta}||_{L^2(\Omega)}=1$ and $\<\delta^2\Es_{\hoc}(0)\^v_{\delta}, \^v_{\delta}\> \leq \L_{\hoc} + \delta$.
We thus obtain
\begin{align}
\L_{\a} \leq \limsup_{N \rightarrow \infty} \L^N_{\a} \leq \limsup_{N \rightarrow \infty} \frac{\<\delta^2\Es^N_{\a}(0)\^v_{\delta}, \^v_{\delta}\>}{||\nabla \^v_{\delta}||^2_{L^2(\Omega)}} = \frac{\<\delta^2\Es_{\hoc}(0)\^v_{\delta}, \^v_{\delta}\>}{||\nabla \^v_{\delta}||^2_{L^2(\Omega)}} \leq \L_{\hoc} + \delta.
\end{align}
The first inequality then follows since $\delta$ can be chosen arbitrarily small.

For the second inequality, the stability constants $\L^N_{\a}$ and $\L_{\cb}$ have explicit characterizations in \cite[Section 3.2]{Hudson:stab} using Fourier transform
\begin{align}
&\L^N_{\a} = \min \{ v^{T} \varphi_{\a}(\kappa) v: \kappa, v \in \SS\},  \quad \text{where} ~
\varphi_{\a}(\kappa) = \frac{h(0)}{2}\sum_{\rho \in \Rc}\frac{sin^2(\frac{\kappa \rho}{2N})}{(\frac{\kappa}{2N})^2}, \nonumber \\
&\L_{\cb} = \min \{ v^{T} \varphi_{\cb}(\kappa) v: \kappa, v \in \SS\},  \quad \text{where} ~
\varphi_{\cb}(\kappa) = \frac{h(0)}{2}\sum_{\rho \in \Rc}\frac{(\frac{1}{2}\kappa \rho)^2}{(\frac{1}{2}\kappa)^2},
\end{align}
where $h(0)=\frac{\partial^2\Es^N_{\a}(0)}{\partial u(0)^2}$ is positive and $\SS:=\{a\in\R: |a|=1\}$.
%\commentlz{what is $h(\rho)$? what is $\kappa$?}
For the higher order continuum model \eqref{hocmodel}, we have
\begin{align}
\L_{\hoc} = \min \{ v^{T} \varphi_{\hoc}(\kappa) v: \kappa, v \in \SS\},  \quad \text{where} ~
\varphi_{\hoc}(\kappa) = \frac{h(0)}{2}\sum_{\rho \in \Rc}\frac{(\frac{\kappa \rho}{2N})^2-\frac{1}{3}(\frac{\kappa \rho}{2N})^4+\frac{2}{45}(\frac{\kappa \rho}{2N})^6}{(\frac{\kappa}{2N})^2}.
\end{align}
We observe that $\varphi_{\hoc}$ is actually the truncated Taylor expansion of $\varphi_{\a}$ up to order 3, while $\varphi_{\cb}$ only preserves the first order term, which indicates the stated result $\L_{\a}\leq \L_{\hoc} \leq \L_{\cb}$ if we assume $N$ is sufficiently large.

\end{proof}

For the stability of the higher order continuum model at small deformations, we have the following Theorem.

\begin{theorem}
If the condition in Lemma \ref{lmm: stability} and \eqref{StabConst:A} are satisfied, we then have
%$$\<\delta^2\Es_{a}(u)v, v\> \geq \frac{1}{2} \ga ||\nabla v||_{L^2}, \qquad \forall v \in \Us^{1,2}.$$
$$\<\delta^2\Es_{\hoc}(u)\^v, \^v\> \geq \1/2 \L_{\hoc} ||\nabla \^v||^2_{L^2(\Omega)}, \qquad \forall \^v \in \Us^{1,2}, u \in \Us^{1,2} \cap \Ks.$$ 
\end{theorem}
\begin{proof}
We note that $u \in \Us^{1,2} \cap \Ks$ can be taken as a perturbation of the reference configuration. Combining the higher order continuum model \eqref{hocmodel} and the Lipschitz continuity of the potential $\phi_{\rho}$, we have
\begin{align}
\big| \<\delta^2\Es_{\hoc}(u)\^v, \^v\> - \<\delta^2\Es_{\hoc}(0)\^v, \^v\>\big|
&\leq \int_{\Omega}\sum_{\rho\in\Rc}\big| \phi''_{\rho}(\rho\nabla u+\frac{\rho^3}{24}\nabla^3u) - \phi''_{\rho}(0)\big|(\rho\nabla \^v)^2 dx  \nonumber \\
&\leq M^{(3, 0)}\kappa||\nabla \^v||^2_{L^2(\Omega)}, \qquad \forall \^v \in \Us^{1,2}.
\end{align}

We finish the proof by choosing $\kappa \leq \L_{\hoc}/(2M^{(3, 0)})$ and applying Lemma \ref{lmm: stability}.
\end{proof}

\section{A priori Error Estimates}
\label{sec: a priori error estimate}

In this section, we present the main result of the a priori error estimate, Theorem \ref{priori}, which essentially shows that the minimizer of the higher order continuum model \eqref{hocmodel} has fourth-order accuracy.

\subsection{Consistency error for the external work}
We first present the following lemma which shows that the approximation error of the external energy is of fourth order.

\begin{lemma}
\label{lemma3}
\textbf{(Consistency error for the external work)}
Suppose $f \in \Us^{1,2}$. Let $\<f, \^v\>_{\ol}$ and $\<f, \^v\>_{\Omega}$ be defined in  \eqref{oldext} and \eqref{cext} respectively. We have the following estimate: 
\begin{align}
\label{extfappr}
|\<f, \~v\>_{\ol} - \<f, \^v\>_{\Omega}| \lesssim ||\nabla^4 f||_{L^2(\Omega)}||\nabla \^v||_{L^2(\Omega)}, \qquad \forall \^v \in \Us^{1,2}.
\end{align}
\end{lemma}
\begin{proof}
By the definition of $\~v(x)$ by \eqref{eq:higherintpcon}, we have
\begin{align}
\<f, \~v\>_{\ol} - \<f, \^v\>_{\Omega} &= \int_{\Omega} \^v(x) \cdot \Big( \sum_{\xi \in \ol}\zeta(\xi- x)f(\xi) - f(x)\Big) dx \nonumber \\ 
&=: \int_{\Omega} \^v(x) \cdot g(x) dx.
\end{align}
By the mean-zero condition for $f$ and the property that $\int \zeta dx = 1$ , it is easy to show that $\int_{\Omega} g(x)dx = 0$.
Hence for an arbitrary constant $c_{\Omega} \in \R$, an application of Cauchy-Schwarz inequality yields that 
\begin{align}
|\<f, \~v\>_{\ol} - \<f, \^v\>_{\Omega}| &= |\int_{\Omega}(\^v(x) - c_{\Omega})\cdot g(x) dx|  \nonumber \\
&\leq ||g||_{L^2(\Omega)}||\^v - c_{\Omega}||_{L^2(\Omega)}.
\end{align}
Choosing $c_{\Omega} = \frac{1}{|\Omega|}\int_{\Omega}\^vdx$ and applying $Poincar\acute{e}$ inequality, we obtain the estimate that
\begin{equation}
\label{eq: consistency of the external energy 1}
||\^v - c_{\Omega}||_{L^2(\Omega)} \lesssim ||\nabla \^v||_{L^2(\Omega)}.
\end{equation}
%Moreover, if we choose $\zeta$ as the cubic spline, by the error estimate of the cubic spline and the fact that $||g||_{L^{2}(\Omega)} \lesssim ||g||_{L^{\infty}(\Omega)}$, we arrive at 
According to the standard Bramble-Hilbert lemma in \cite[Theorem 6.4]{Braess2007}, we can estimate the $L^2$-norm of $g$ by
\begin{equation}
\label{eq: consistency of the external energy 2}
||g||_{L^2(\Omega)} \lesssim ||\nabla^4 f||_{L^2(\Omega)},
\end{equation}
which can also be obtained by a direct extension of \cite[Lemma 13]{AlexIntp:2012}. Combination of \eqref{eq: consistency of the external energy 1} and \eqref{eq: consistency of the external energy 2} leads to the required result.
\end{proof}

\subsection{A priori error estimate}
\label{subsec: A priori error estimate}

We first state the well-known inverse function theorem \cite[Lemma 2.2]{Ortner:qnl.1d}. 

\begin{lemma}\textbf{(Inverse function theorem)}  Let $\As, \Bs$ be Banach spaces, $\Os$ an open set of $\As$, and let $\Fs : \Os \rightarrow \Bs$ be Fr\'echet differentiable with Lipschitz-continuous derivative $\delta \Fs$:
$$||\delta \Fs(U) - \delta \Fs(V)||_{L(\Bs, \As)} \leq M ||U - V||_{\As} \qquad \forall U, V \in \As,$$
where $M$ is a Lipschitz constant. Let $X \in \Os$ and suppose also that there exists $\eta, \sigma > 0$ such that
$$\quad ||\Fs(X)||_{\Bs} \leq \eta, \quad ||\delta \Fs(X)^{-1}||_{L(\Bs, \As)} \leq \sigma,$$
%$$||\Fs'(U) - \Fs'(V)||_{L(\Bs, \As)} \leq \omega(||U - V||_{\As}) \qquad \forall ||U||_{\As}, ||V||_{\As} \leq 2\eta\sigma,$$
%$$2\sigma\omega(2\eta\sigma)\leq 1, \quad \sigma\omega(2\eta\sigma)<1.$$
%Then there exists a unique $U\in \As$ such that $\Fs(U)=0$ and $||U||_{\As}\leq 2\eta\sigma$. 
$$2M\eta\sigma^{2}<1.$$
Then there exists a locally unique $Y \in \As$ such that $\Fs(Y)=0$ and $||Y - X||_{\As}\leq 2\eta\sigma$. 
\label{IFT} 
\end{lemma}

%\begin{lemma}\textbf{(Inverse function theorem)} 
%\label{lmm: inverse function theorem}
%Let $\As$ be a subspace of a Banach space equipped with norm $||\cdot||_{\As}$, $\Bs$ is a Banach space, and let $\Fs : \As \rightarrow \Bs$ be Fr\'echet differentiable with Lipschitz-continuous derivative $\delta \Fs$:
%$$||\delta \Fs(U) - \delta \Fs(V)||_{L(\Bs, \As)} \leq M ||U - V||_{\As} \qquad \forall U, V \in \As,$$
%where $M$ is a Lipschitz constant. Let $X \in \As$ and suppose that there exists $\eta, \sigma > 0$ such that
%$$ ||\Fs(X)||_{\Bs} \leq \eta, \quad ||\delta \Fs(X)^{-1}||_{L(\Bs, \As)} \leq \sigma,$$
%$$2M\eta\sigma^{2}<1.$$
%Then there exists a locally unique $Y \in \As$ such that $\Fs(Y)=0$ and $||Y - X||_{\As}\leq 2\eta\sigma$. 
%\label{IFT} 
%\end{lemma}

The following theorem shows that for a stable and sufficiently small deformation, the solution of the higher order continuum model \eqref{hocmodel} is a good approximation to the solution of the atomistic model.

\begin{theorem}\label{priori}
\textbf{(A priori error estimate)}
Suppose \eqref{StabConst:A}  is satisfied and $u^\a$ is a strongly stable atomistic solution of \eqref{Aprob}. $\Pi$ is defined in \eqref{eq: 5th order finite difference} and $f \in \Us^{1,2}$. If we assume that $\eta_1, \eta_2$ are sufficiently small such that $||\nabla^j \Pi u^{\a}||_{L^2(\Omega)} \leq \eta_1$, $j=2,3,4,5$, and $||\nabla^4 f||_{L^2(\Omega)} \leq \eta_2$, there exists a stable solution $u^{\hoc}$ of problem \eqref{HOCprob} in $W^{1, 2}$ such that
\begin{align}
||\nabla \Pi u^{\a} - &\nabla u^{\hoc}||_{L^2(\Omega)} \leq C (||\nabla^5 \Pi u^{\a}||_{L^2(\Omega)} + ||\nabla^2\Pi u^{\a}\nabla^4\Pi u^{\a}||_{L^2(\Omega)} \nonumber \\ 
+& ||\nabla^3\Pi u^{\a}(\nabla^2\Pi u^{\a})^2||_{L^2(\Omega)} + ||\nabla^3\Pi u^{\a}||^2_{L^4(\Omega)} + ||\nabla^2\Pi u^{\a}||^4_{L^8(\Omega)} + ||\nabla^4 f||_{L^2(\Omega)}),
\label{ap}
\end{align}
where $C$ depends on $\L_{\hoc}$,  $M^{(j, 4)}$, $j=2,...,5$.
\end{theorem}
\begin{proof}
%We want  to use the regularity of $\Pi u^\a$ to give a upper bound of $||\nabla I u^\a - \nabla u^{hoc} ||_{L^2}$, which is different from \S~\ref{sec: mainresults}. Firstly, we should assume that the atomistic model satisfies the strongly stable condition,
%$$\langle\delta^2 \Es_{a}(0)v,v\rangle \geq \bar{c}||\nabla v||^2_{L^2}, \quad \forall v \in \Uc,$$
%$where $\bar{c}>0$ and a priori error estimate can be divided into three steps.
Following the framework of the a priori error estimates in \cite{bqce, LuOr:acta, OrtnerWang:2011}, we divide the proof into three steps. Recalling the definition of the space $\Us^{1,2}$ in Section \ref{Sec: variational problems}, we apply Lemma \ref{IFT} with $\As:=\Us^{1,2}, \Os:= \Ks, \Bs:=W^{-1,2}, X:=\Pi u^\a$. We define the operator $\Fs : \Os \rightarrow \Bs$ by 
$$\<\Fs(u), \^v\> := \<\delta\Es_{\hoc}(u), \^v\> - \<f, \^v\>_{\Omega} \qquad \forall \^v \in \Us^{1,2}, u \in \Os.$$
$\delta \Fs$ is Lipschitz continuous due to the Lipschitz continuity of the potential $\phi_{\rho}$. We also note that $X \in \Os$ since $u^\a \in \Uc^{1,2} \cap \Kc$ and $\eta_1$ is chosen to be sufficiently small.

\textbf{Step 1: Stability.}~~~~
In Section \ref{sec:stability}, we have already shown the stability of the higher order continuum model, that is,
$$\langle \delta^2 \Es_{\hoc}(\Pi u^{\a})\^v, \^v \rangle \geq \1/2 \L_{\hoc} ||\nabla \^v||^2_{L^2(\Omega)}, \quad \forall \^v \in \Us^{1,2}, $$
where $\L_{\hoc}$ is positive. Hence, we have
$$||\delta\Fs(X)^{-1}||_{L(\Bs, \As)} \leq (\1/2 \L_{\hoc})^{-1} =: \sigma.$$

\textbf{Step 2: Consistency.}~~~~
We shall also require $\Fs(X)$ to be consistent. In fact, if we put the interpolation of the atomistic solution $u^{\a}$ in $\Es_{\hoc}$, we have
\begin{align}
\<\Fs(X), \^v\> =&\< \delta \Es_{\hoc}(\Pi u^{\a}), \^v\> - \<f, \^v\>_{\Omega}\nonumber \\ 
=&\< \delta \Es_{\hoc}(\Pi u^{\a}), \^v \> - \<f, \^v\>_{\Omega} - \< \delta \Es_{\a}(u^{\a}), \~v\> + \<f, \~v\>_{\ol} \nonumber \\ 
%=& \Big[ \<\delta \Es_{\hoc} (I u^{\a}), \^v\> - \<\delta\Es_{\hoc}(\Pi u^{\a}), \^v \>\Big] \\ \nonumber
=& \Big[ \<\delta\Es_{\hoc}(\Pi u^{\a}), \^v \> -  \< \delta \Es_{\a}(u^{\a}), \~v\>\Big] - \Big[\<f, \^v\>_{\Omega} - \<f, \~v\>_{\ol}\Big] \nonumber \\ 
%=&\int_{\R} (S^{hoc}(I u^\a; x) - S^{hoc}(\Pi u^\a; x)) \cdot \nabla \^v dx + \int_{\R} R(\Pi u^\a;x)\cdot \nabla \^v dx -\Big[ \<f, \^v\>_{\R} - \<f, v\>_{\Lambda*}\Big] \\ \nonumber
:=&  T_1 + T_2,
\end{align}
where we decompose the consistency into two parts:  $T_1$ is the modeling error and $T_2$ is the consistency error of the external force. Applying Theorem \ref{thm2} and Lemma \ref{lemma3}, we have
\begin{align}
||\Fs(X)||_{\Bs} \leq \eta &:= C(||\nabla^5 \Pi u^{\a}||_{L^2(\Omega)} + ||\nabla^2\Pi u^{\a}\nabla^4\Pi u^{\a}||_{L^2(\Omega)} \nonumber \\ 
+& ||\nabla^3\Pi u^{\a}(\nabla^2\Pi u^{\a})^2||_{L^2(\Omega)} + ||\nabla^3\Pi u^{\a}||^2_{L^4(\Omega)} + ||\nabla^2\Pi u^{\a}||^4_{L^8(\Omega)} + ||\nabla^4 f||_{L^2(\Omega)}),
\end{align}
where $C$ depends on $M^{(j, 4)}$, $j=2,...,5$. 

\textbf{Step 3: Inverse function theorem.}~~~~
Combing the stability result in Step 1 and the consistency result in Step 2 and applying Lemma \ref{IFT}, we obtain the existence of the solution of the higher order continuum model $u^{\hoc}$ in $W^{1, 2}$
and the error estimate
$$||\nabla \Pi u^{\a} - \nabla u^{\hoc}||_{L^2(\Omega)}\leq \frac{4\eta}{\L_{\hoc}},$$
which can be guaranteed if we choose $\eta_1, \eta_2$ to be sufficiently small.
\end{proof}

\begin{remark}[Scaling]
\label{remark: scaling}
Up to now, we have taken the unit interatomic spacing in the reference lattice. In order to illustrate the order of accuracy, we scale the interatomic spacing by $\eps$, that is, $X := \eps x$, $U := \eps u$ and $F := \eps^{-1} f$. Reversing the scaling, we have $u^\a(x) := \eps^{-1} U^\a(\eps x)$ of the atomistic problem \eqref{Avarprob} and the external force $f(x) := \eps F(\eps x)$. It can be easily shown that $||\nabla^j \Pi u^{\a}||_{L^2(\Omega)} = \eps^{j-1-\1/2} ||\nabla^j \Pi_{\eps} U^{\a}||_{L^2(\Omega)}$, $j=1,2,3,4,5$, where $\Pi_{\eps}$ is the same interpolation as $\Pi$ under $\eps$ scale. We also scale the estimate \eqref{eq: consistency of the external energy 2} by $||\nabla^j f||_{L^2(\Omega)} = \eps^{j + 1 -\1/2} ||\nabla^j F||_{L^2(\Omega)}$, $j=0,1, 2, 3, 4$. Hence, Theorem \ref{priori} essentially shows that our higher order continuum model is of fourth order accuracy, namely, 
\begin{align}
||\nabla \Pi_{\eps} U^{\a} - &\nabla U^{\hoc}||_{L^2(\Omega)} \leq C\eps^4 (||\nabla^5 \Pi_{\eps} U^{\a}||_{L^2(\Omega)} + ||\nabla^2\Pi_{\eps} U^{\a}||_{L^2(\Omega)}||\nabla^4\Pi_{\eps} U^{\a}||_{L^2(\Omega)} \nonumber \\ 
+& ||\nabla^3\Pi_{\eps} U^{\a}||_{L^2(\Omega)}||\nabla^2\Pi_{\eps} U^{\a}||^2_{L^4(\Omega)} + ||\nabla^3\Pi_{\eps} U^{\a}||^2_{L^4(\Omega)} + ||\nabla^2\Pi_{\eps} U^{\a}||^4_{L^8(\Omega)} + ||\nabla^4 F||_{L^2(\Omega)}).
\label{apeps}
\end{align}  
\end{remark}

\begin{remark}[higher regularity]
In Theorem \ref{priori}, we only prove the existence of the minimizer in $W^{1,2}$, although the approximation space $\Us^{1, 2} \cap \Ks$ is more restrictive. It is possible to prove higher regularity ($W^{5,2}$) given ellipticity of the Euler-Lagrange equation, similar as the $W^{3,2}$ regularity of the Cauchy-Born solution as in \cite[Proof of Theorem 3]{OrtnerTheil2012}. In this paper, we are mainly concerned with the $W^{1,2}$ accuracy for the solution of the higher order continuum model.
\label{rem:highregularity}
\end{remark}

Similar to \cite[Proposition 3.2]{bqce}, we also have the estimate for the error in energy $|\Es_{\a}(\Pi u^{\a}) - \Es_{\hoc}(u^{\hoc})|$.

\begin{theorem}\label{priorienergy}
\textbf{(Energy estimate)}
Under the conditions of Theorem \ref{priori}, we have
\begin{align}
\big| \Es_{\a}(\Pi u^{\a}) - \Es_{\hoc}(u^{\hoc})\big| &\leq C (%||\nabla \Pi u^{\a} - \nabla u^{\hoc}||^2_{L^2(\Omega)}+ 
||\nabla^2\Pi u^{\a}||_{L^2(\Omega)}||\nabla^4\Pi u^{\a}||_{L^2(\Omega)} %\nonumber \\ 
+ ||\nabla^3\Pi u^{\a}||_{L^2(\Omega)}||\nabla^2\Pi u^{\a}||^2_{L^4(\Omega)} \nonumber \\
+& ||\nabla^5\Pi u^{\a}||_{L^2(\Omega)} + ||\nabla^3\Pi u^{\a}||^2_{L^4(\Omega)} + ||\nabla^2\Pi u^{\a}||^4_{L^8(\Omega)}),
\label{ap}
\end{align}
where $C$ depends on $M^{(j, 3)}$, $j=2,...,5$.
\end{theorem}

\section{Numerical Experiments}
\label{sec:numerics}

\def\ds{\displaystyle}
We present two numerical experiments to illustrate the analytical results of this paper. We include up to second nearest neighbor interactions in our energy functional such that 
\begin{equation}
 \Es_\a(u):=\sum_{\xi \in \ol}(\phi_{1}(D_{1}u(\xi)) + \phi_{2}(D_{2}u(\xi))),
\label{eq: NN A model}
\end{equation}
and
\begin{equation}
\Es_\hoc(u):=\int_{\Omega}\big( \phi_{1}(\nabla u+\frac{1}{24}\nabla^3u)+\phi_{2}(2\nabla u+\frac{1}{3}\nabla^3u)\big) dx.
\label{eq: NN hocmodel}
\end{equation}
We set the computational domain to be $\Omega = [-1,1]$. In the reference configuration, there are $2N+1$ equally distributed atoms in $\Omega$, hence the scaling parameter is $\eps := 1/N$. We choose
\begin{equation}
f(x) = \cos(\pi x) 
\end{equation}
as the external force so that a nonlinear but small enough displacement (or equivalently deformation) is generated.  It is easy to see that $f \in \Us^{1,2}$. We will carry out numerical experiments for both the harmonic potential $\ds\phi(r)=\1/2 \big(\frac{r}{\eps}-1\big)^2$ and the Leonard-Jones potential $\ds\phi(r)=\big(\frac{r}{\eps}\big)^{-12}-2\big(\frac{r}{\eps}\big)^{-6}$.

%As for boundary condition, since our discussion is in one dimension and we only care about the approximation error, so we want to eliminate the effects of boundary conditions as much as possible. Hence, we make use of periodic boundary condition in our numerical experiments since it is easy to see the approximation order. Here we remark that the reason why we do not use Dirichlet boundary condition is that in \cite{}, which is out of discussion in this paper. We will discuss it in our future work.
We use $C^3$ finite element to solve the variational form \eqref{eq:HOCfv} of higher order continuum model. We denote the positions of the atoms to be $\{x_i\}^{2N+1}_{i=1}$, and let the nodes of the finite elements coincide with the atoms so that $\Omega$ is partitioned by $\Ts:=\{T_i\}_{i=1}^{2N+1}$ where $T_i:=[x_{i-1}, x_{i}]$. The finite element solution space is then defined by
\begin{align*}
\Us^{1,2}_{\eps} := \{ u \in C^3(\Omega) \cap \Us^{1,2} \big| u |_{T_i} \in \P_5 \},
\end{align*}
where $\P_5$ is the quintic polynomial function space. We search the approximate local minimizer of $\Es_\hoc(u)$ defined by \eqref{eq: NN hocmodel} in $\Us^{1,2}_{\eps}$ using BFGS algorithm. 

We use Guass-Legendre quadrature of order 5 to approximate the integral $\int_{\Omega}fudx$ which is the external work in the higher order continuum model to make the quadrature error negligible compared to the consistency error given in \eqref{extfappr}.

We use the following protocol to quantify the modeling error:
\begin{enumerate}
\item Let $\eps_1=2^{-3}$, $\eps_2=2^{-4}$, ... , $\eps_8=2^{-10}$ be the interatomic spacing which also define the reference lattice and the finite element mesh. 
\item Compute the atomistic solution $u^{\a,\eps_i}$ and the higher order continuum solution $u^{\hoc,\eps_i}$ on different lattices (or corresponding meshes). 
\item Compute the error $\| \nabla I u^{\a,\eps_i}- \nabla u^{\hoc,\eps_i} \|_{L^2}$, where $I$ represents a smooth interpolation operator such that the interpolation error is negligible compared to the consistency error. 
\end{enumerate}

\begin{remark}
The numerical error for the higher order continuum model is fifth order given the approximation space $\Us^{1,2}_{\eps}$ and the higher regularity of $u^{\hoc}$ (see Remark \ref{rem:highregularity}), we have,
 $$||\nabla u^{\hoc, \eps} - \nabla u^{\hoc}||_{L^2} \lesssim \eps^5.$$
 \end{remark}
% We have various choices of finite element methods to solve the higher order continuum model \eqref{hocmodel}, for example, spline function, but the numerical error can not be less than forth order accuracy.

\subsection{The harmonic potential}

We first give a numerical justification of our main result for the harmonic potential $\ds\phi(r)=\1/2 \big(\frac{r}{\eps}-1\big)^2$. Fixing $\eps=2^{-3}$, we compute the solutions of the atomistic model \eqref{Aenergyinu}, Cauchy-Born model \eqref{cbmodel}, and higher order continuum model \eqref{hocmodel}. 

In Figure \ref{u}, we observe that the solution of the higher order continuum model ($\rhd$) is closer to the solution of the atomistic model ($\circ$) than that of the Cauchy-Born model ($\bullet$). 

\begin{figure}[h!]
  \centering
  \includegraphics[scale=0.53]{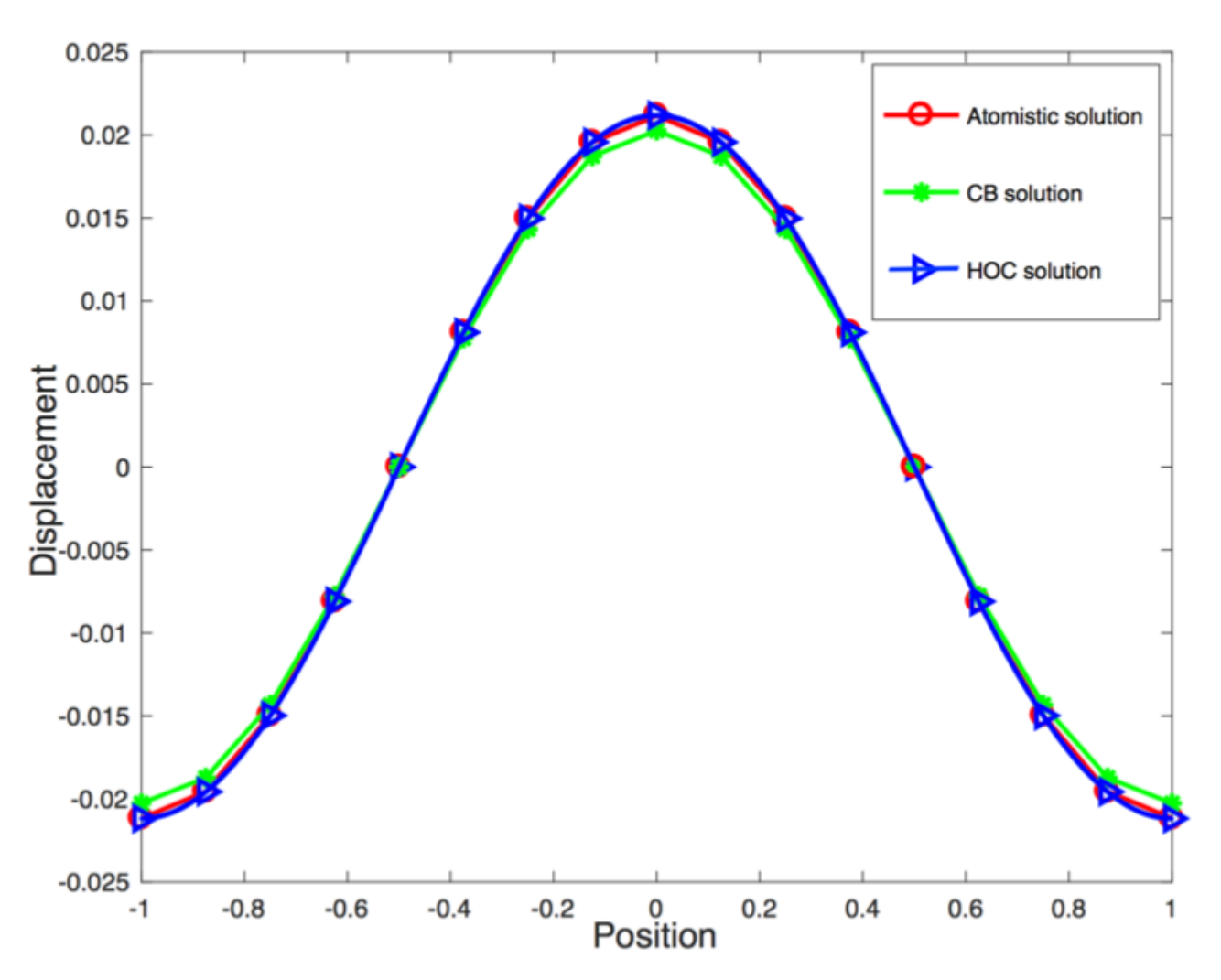}
  \caption{Displacement}
  \label{u}
\end{figure}

We then compute and plot the error $\| \nabla I u^{\a,\eps}- \nabla u^{\hoc,\eps} \|_{L^2}$ in Figure \ref{figs:mde}, notice that the modeling error $\|\nabla \Pi u^{\a} - \nabla u^{\hoc}\|_{L^2}$ is the dominant part in $\| \nabla I u^{\a,\eps}- \nabla u^{\hoc,\eps} \|_{L^2}$. Figure \ref{figs:mde} clearly shows the fourth order accuracy of the higher order continuum model, compared with second order accuracy of the Cauchy-Born model. We need to mention here the importance of the proper choice of the interpolation operator $I$ for the atomistic solution $u^{\a,\eps}$. 

%Recall that in our analysis, an interpolation based on higher order finite difference $\Pi u^{\a,\eps}$ is employed. In practice, however, we do not use $\Pi u^{\a,\eps}$ since it is a piecewise polynomial of degree nine which is difficult to compute. Alternatively, we choose an interpolation (or a spline) $I u^{\a,\eps}$ so that 
%\begin{equation}
%\| \nabla I u^{\a,\eps} - \nabla \Pi u^{\a,\eps}\| \lesssim \eps^4.
%\label{eq: fourth order interpolation of Pi u}
%\end{equation}
%Then by the triangle inequality we have
%\begin{equation}
%\| \nabla I u^{\a,\eps}- \nabla u^{\hoc,\eps}\| \le \| \nabla I u^{\a,\eps} - \nabla \Pi u^{\a,\eps}\| + \| \nabla \Pi u^{\a,\eps}  - \nabla u^{\hoc,\eps}\| \lesssim \eps^4,
%\label{eq: triangle inequality for interpolations}
%\end{equation} 
%which essentially means that the fourth order model error is preserved. We indeed observe the correct order of the model error in Figure \ref{figs:mde} when, for example, fourth order Hermitian interpolation or quartic splines are used. On the other hand, if we choose an interpolation such that \eqref{eq: fourth order interpolation of Pi u} is not satisfied, then the error will be dominated by the first term on the right hand side of \eqref{eq: triangle inequality for interpolations} which prevent us from observing the right order. The example of using cubic spline for the atomistic solution is also shown in Figure \ref{figs:mde} to illustrate such problem. 

We use different interpolation operators $I$ in Figure \ref{figs:mde}, interpolation operator $\Pi$ and quartic splines interpolation can preserve the 4th order accuracy, while cubic spline interpolation gives suboptimal results (3rd order accuracy).

\begin{figure}[h!]
  \centering
  \includegraphics[scale=0.45]{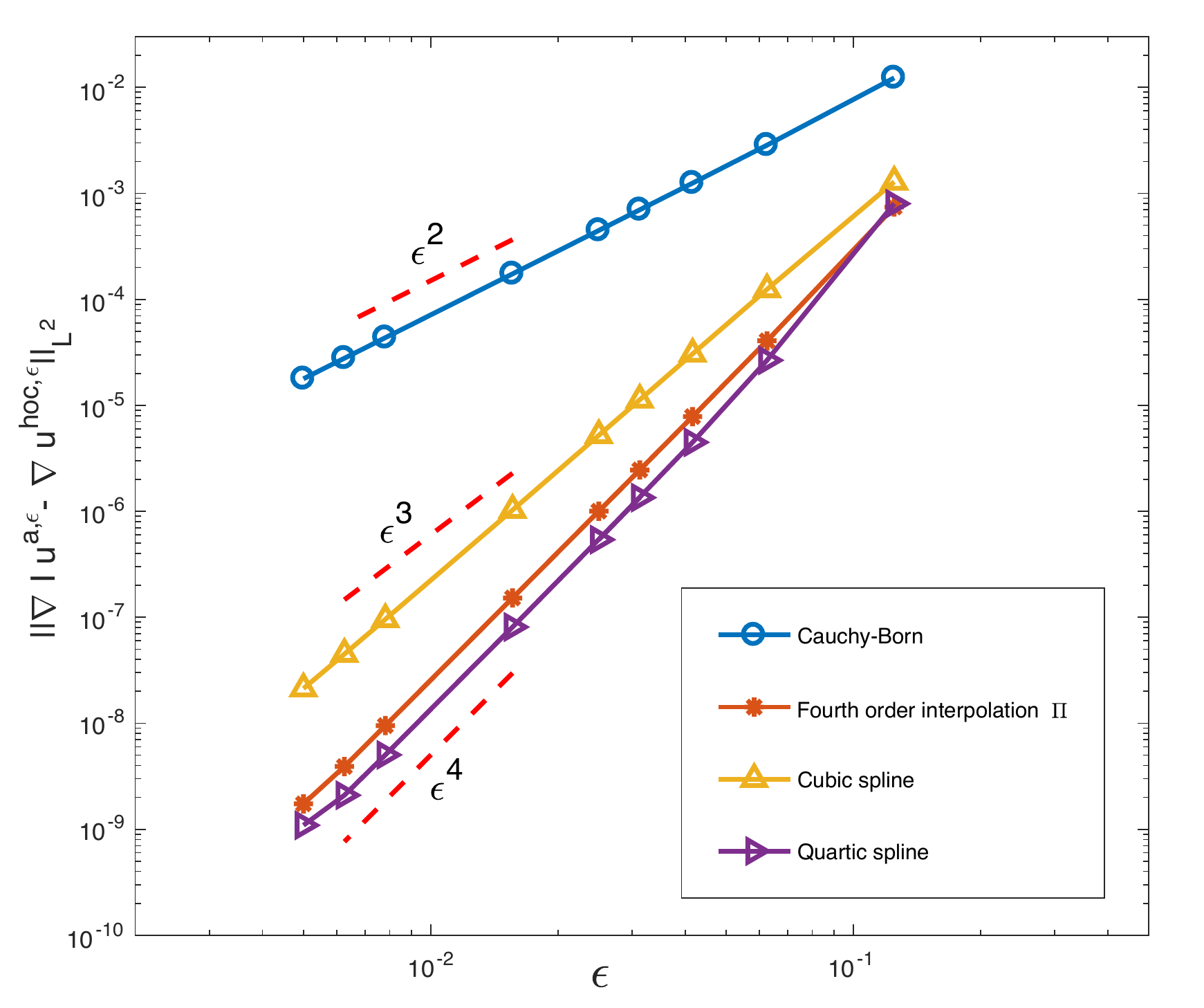}
  \caption{Modeling error with different $I$ in linear case}
  \label{figs:mde}
\end{figure}

%\subsubsection{Periodic boundary condition in nonlinear case}
\subsection{Lennard-Jones Potential}
Our second numerical example is for Lennard-Jones potential $\ds\phi(r)=\big(\frac{r}{\eps}\big)^{-12}-2\big(\frac{r}{\eps}\big)^{-6}$ with the interpolation $I u^{\a,\eps}$ being the quartic spline. Figure \ref{HOCmdemin} shows that the order of the modeling error is not affected by the nonlinearity of the potential.

We also plot the error in energy in Figure \ref{HOCmdeenergy}. We see that the error in energy is of fourth order for the higher order continuum model compared with the second order for the Cauchy-Born model, which is consistent with Theorem \ref{priorienergy}. 

\begin{figure}[h!]
  \centering
  \includegraphics[scale=0.45]{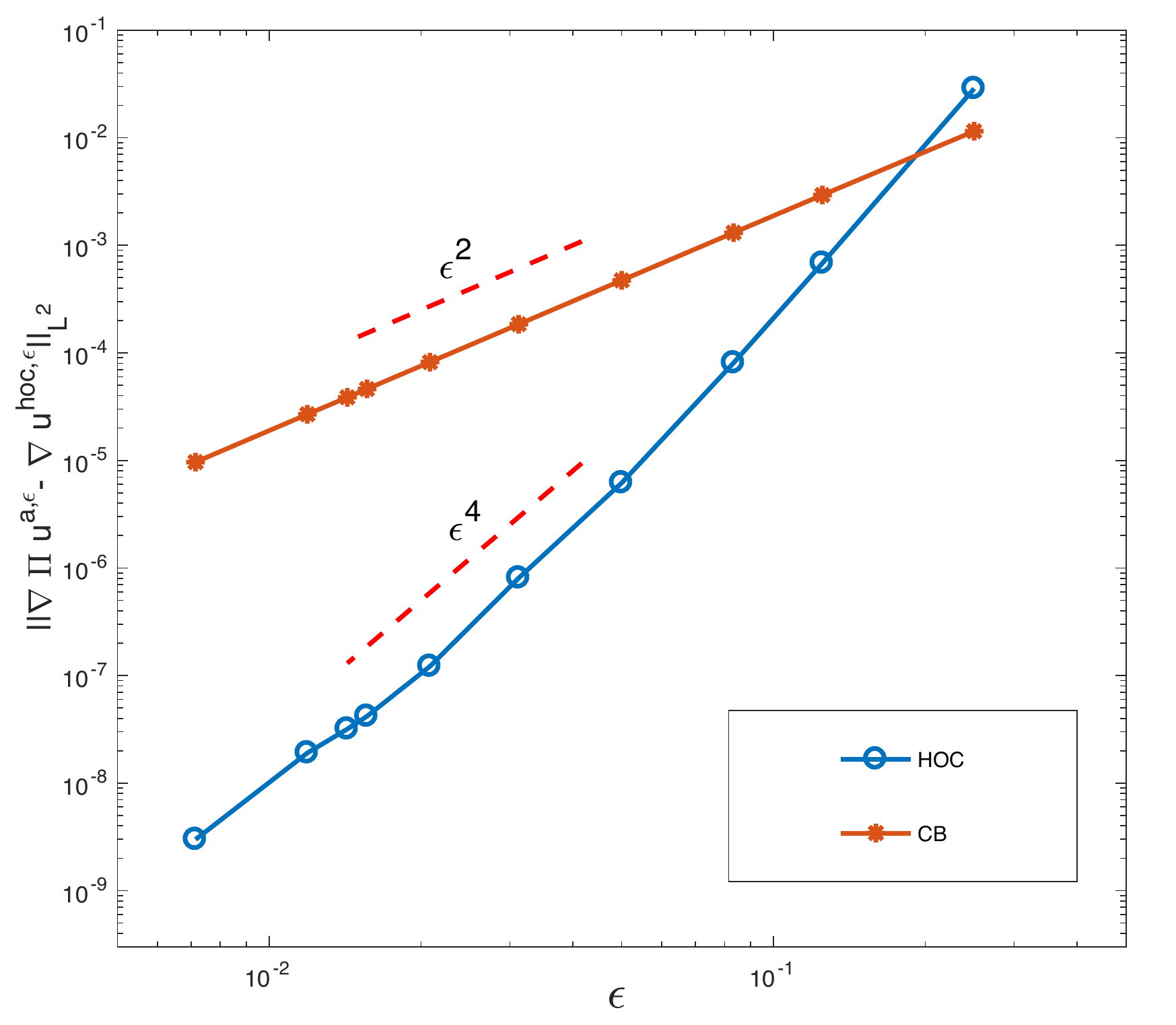}
  \caption{Modeling error for Lennard-Jones interactions}
  \label{HOCmdemin}
\end{figure}

\begin{figure}[h!]
  \centering
  \includegraphics[scale=0.48]{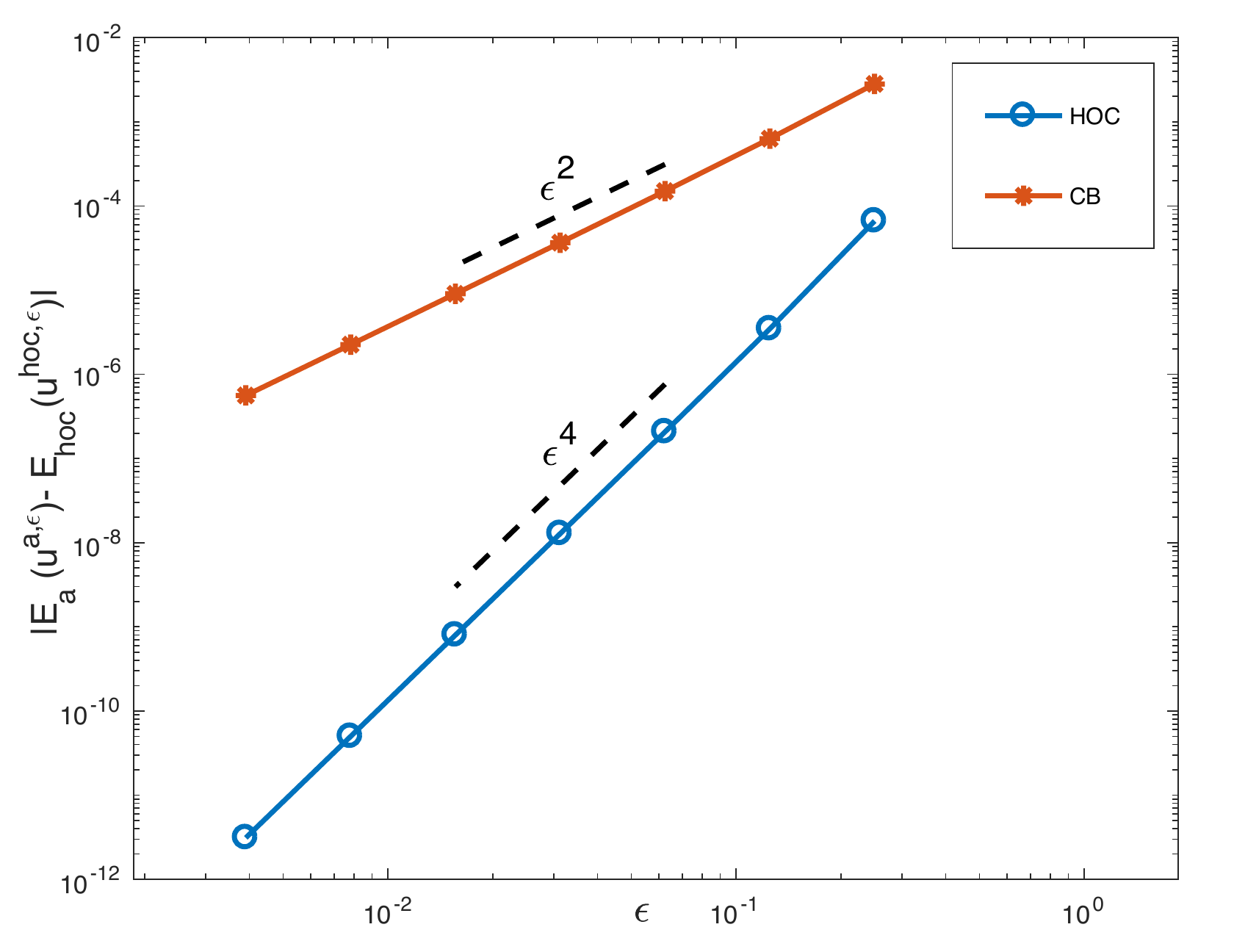}
  \caption{Error in Energy}
  \label{HOCmdeenergy}
\end{figure}

\section{Conclusion and future work}
\label{sec:conclusion}
In this paper, we derive a higher order nonlinear elasticity model from the atomistic model in one dimension, and present a rigorous a priori error analysis for this higher order continuum model. Using the techniques developed in \cite{Ortner:qnl.1d, OrtnerTheil2012, bqce}, we prove that the modeling error of our higher order continuum model is fourth order, compared with the second order accuracy of the well-known Cauchy Born model. Numerical experiments are carried out to verify our theoretical results.

This work opens up several interesting research directions:

The first direction is the extension of the current work to general multibody interactions. We note that the inner expansion technique in \cite{arndt2005derivation} does not apply in this case, since it requires that the site potential $\phi_{\bar{\xi}}({y(\xi)}_{\xi\in\L})$ for the site $\bar{\xi}$ can be written as
$$\phi_{\bar{\xi}}({y(\xi)}_{\xi\in\L}) = \psi(\sum_{\xi\in\L}a_{\xi}y(\xi)),$$
where $\psi$ is a function and $a_{\xi}$ are constants.
 In fact, the essential problem here, is how to determine the "optimal" expansion points in the Taylor expansion of the energy functional. While it is natural to use the midpoints of the bonds as the expansion points for pair interactions, multibody interactions involve a number of bonds and thus have cross terms in the Taylor expansion. This may lead to various possible formulations of the higher order continuum model corresponding to different choices of expansion points. We need to choose the expansion points "optimally" such that the cancellation in \eqref{Tylorexp} can be achieved at the highest possible order. 

The second direction is the construction and analysis of the higher order continuum model in higher dimensions. Such extension seems to be straightforward following the framework proposed in the current work. However, we note that Lemma \ref{lemma1}, which serves as the key to the cancellation of the lower order terms in \eqref{Tylorexp}, in general does not hold for $k>1$ in higher dimensions. To be more precise, in higher dimensions, the tensor product such as $\rho \otimes \rho$ will appear in the stress difference $R(u;x)$. Therefore an alternative identity for the cross terms should be sought after in higher dimensions to guarantee the accuracy. 

The third direction is the development of atomistic/continuum coupling with the higher order continuum model. The key to design "optimal" coupling method is to balance the modeling error with coarsening and truncation errors, within the analytical framework in \cite{ML_CO_BK_Energy_Blended_QC_2011, BK_ML_Energy_Blended_QC_2011, bqce, 2016bgfc, 2012grac}.  The reduction of modeling error by the higher order models can facilitate the construction of coupling method with quasi-optimal convergence rate. In particular, for complex lattices, the modeling error of Cauchy Born is only first order due to lack of symmetry and becomes the bottleneck for the coupling method. We expect higher order models can be used to alleviate this problem. Once the coupling model is developed, the study of adaptivity should be under way where \cite{AM_A_Post_ML_FK_Model_2007, AM_A_Post_ML_FK_Model_2008_Model, CO_HW_ACC_1D_A_Posteriori, HW_ML_PL_LZ_Adaptiv_AC_2D, HW_SY_2018_Efficiency_A_Post_1D, HW_SL_FY_A_Post_BQCF_2018} should provide good references.

\section*{Acknowledgements}
\section*{Funding}
YW and LZ were supported by NSFC grant 11871339, 11861131004, 11571314. HW was supported by National Science Foundation China grant No.11501389, 11471214. 

%\clearpage
\appendix
\section*{Appendix A. Derivation and Justification of the higher order continuum model with sixth order accuracy}
\label{app1}
In this section, we present the derivation and the theoretical justification of the higher order continuum model which contains $\nabla u$, $\nabla^3 u$ and $\nabla^5 u$ with 6th order consistency. The stability and convergence of this model can be verified similar as 4th order model \eqref{hocmodel}. 

Following the derivation introduced in Section \eqref{sec:continuum}, we extend \eqref{eq: Taylor expansion of u} by truncating the terms whose orders are higher than 5, we have
\begin{align}
&D_{\rho}u(\xi) = u(\xi+\rho) - u(\xi)  \nonumber \\
=& \Big[u(\xi') + \frac{\rho}{2}\nabla u(\xi') + \frac{1}{2}\nabla^2u(\xi')(\frac{\rho}{2})^2 + \frac{1}{6}\nabla^3u(\xi')(\frac{\rho}{2})^3 + \frac{1}{24}\nabla^4u(\xi')(\frac{\rho}{2})^4 + \frac{1}{120}\nabla^5u(\xi')(\frac{\rho}{2})^5 + ...\Big] \nonumber  \\ 
&- \Big[u(\xi') - \frac{\rho}{2}\nabla u(\xi') + \frac{1}{2}\nabla^2u(\xi')(\frac{\rho}{2})^2 - \frac{1}{6}\nabla^3u(\xi')(\frac{\rho}{2})^3 + \frac{1}{24}\nabla^4u(\xi')(\frac{\rho}{2})^4 - \frac{1}{120}\nabla^5u(\xi')(\frac{\rho}{2})^5 + ...\Big] \nonumber \\ 
\approx& \rho\nabla u(\xi')+\frac{\rho^3}{24}\nabla^3u(\xi')+\frac{\rho^5}{1920}\nabla^5u(\xi').
\label{eq: Taylor expansion of u to 5}
\end{align}
The corresponding higher order continuum model reads
\begin{equation}
\Es^{*}_\hoc(u)=\int_{\Omega}\sum_{\rho\in \Rc} \phi_{\rho}(\rho\nabla u+\frac{\rho^3}{24}\nabla^3u+\frac{\rho^5}{1920}\nabla^5u) dx.
\label{hocmodel6}
\end{equation}
The energy functional \eqref{hocmodel6} depends on $\nabla u$, $\nabla^3 u$ and $\nabla^5 u$, and we will show it has 6th order consistency. We introduce the space by imposing periodic boundary condition and mean zero condition on it:
\begin{equation}
  \^{\Us}^{1,2} := \big\{ u \in W^{6,2}:
  \nabla^{j} u(x+2N) = \nabla^{j} u(x), j = 0,1,2,3,4,  {\textstyle \int_{\Omega} u dx = 0}  \big\}.
  \label{def: Displacement Space5}
\end{equation}
The first variation of the higher order continuum energy functional $\Es^{*}_{\hoc}$ is given by
\begin{equation}
\<\delta\Es^{*}_{\hoc}(u),\^v\>=\int_{\Omega}\sum_{\rho \in \Rc}\phi'_{\rho}(\rho \nabla u+\frac{\rho^3}{24}\nabla^3u+\frac{\rho^5}{1920}\nabla^5u)(\rho \nabla \^v+\frac{\rho^3}{24}\nabla^3 \^v+\frac{\rho^5}{1920}\nabla^5\^v)dx, \quad \forall \^v \in \^{\Us}^{1,2}.
\label{eq:HOCfv5}
\end{equation}
Using the integration by parts and the periodic boundary condition of the test function, we obtain
\begin{align}
\label{firstvariation5}
\<\delta\Es^{*}_{\hoc}(u),\^v\>
:=& \int_{\Omega}S^{\hoc}_{*}(u;x)\cdot\nabla \^v(x)dx,% \qquad \forall \^v \in \Us^{1,2},
\end{align}
where %$S^{\hoc}_{*}(u; x)$ will be calculated later.
\begin{align}\label{25}
S^{\hoc}_{*}(u;x) = &\sum_{\xi\in\Lambda}\sum_{\rho \in \Rc} \Big[ \rho\phi'_{\rho}(\nabla_{\rho} u)+\frac{\rho^4}{12}\phi''_{\rho}(\nabla_{\rho} u)\nabla^3u+\frac{\rho^5}{24}\phi'''_{\rho}(\nabla_{\rho} u)(\nabla^2u)^2 \nonumber \\ 
%%%%%%%%%%%%%%
&+\frac{\rho^6}{360}\phi''_{\rho}(\nabla_{\rho}u)\nabla^5u+\frac{\rho^7}{240}\phi'''_{\rho}(\nabla_{\rho}u)(\nabla^3u)^2+\frac{\rho^7}{180}\phi'''_{\rho}(\nabla_{\rho}u)\nabla^2u\nabla^4u \nonumber \\ 
%%%%%%%%%%%%%
&+\frac{7\rho^8}{1440}\phi^{(4)}_{\rho}(\nabla_{\rho}u)\nabla^3u(\nabla^2u)^2 + \frac{\rho^9}{1920}\phi^{(5)}_{\rho}(\nabla_{\rho}u)(\nabla^2u)^4 \Big] \chi_{\xi,\rho}(x).
\end{align}

Following the analysis in Section \ref{sec:modelingerror}, we give a pointwise sixth-order consistency estimate of the stress error $R(u;x)=S^{\a}(u;x)-S^{\hoc}_{*}(u;x)$. The key step is the cancellation of the terms whose order is less than six in \eqref{Tylorexp5}, which can be achieved by the extension of Lemma \ref{lemma1} to $k=5$. If we choose the basis function $\zeta$ as the quintic spline basis function, then it preserves the fifth order polynomials. Following the proof of Lemma \ref{lemma1}, we have
\begin{equation}\label{i5}
\sum_{\xi \in \Lambda}\chi_{\xi,\rho}(x)(\xi-x)^k=\frac{(-\rho)^k}{k+1}, \qquad  k=0,1,2,3,4,5.
\end{equation}
The atomistic stress is the same as \eqref{eq:StressA}:
\begin{equation}
S^{\a}(u;x) := \sum_{\xi\in\Lambda}\sum_{\rho \in \Rc}(\rho\phi'_{\rho}(D_{\rho}u(\xi)))\chi_{\xi,\rho}(x).
\label{eq:StressA5}
\end{equation}

We then expand $D_{\rho}u(\xi)$ at $x$ for $\rho \in \Rc$ in $v_x$, which is the neighbourhood of $x$.
\begin{align}\label{35}
D_{\rho}u(\xi) =& \nabla_{\rho}u + \Big[ \rho(\xi-x) + \frac{\rho^2}{2}\Big]\nabla^2u + \Big[ \frac{\rho}{2}(\xi-x)^2 + \frac{\rho^2}{2}(\xi-x) + \frac{\rho^3}{6} \Big]\nabla^3u \nonumber \\ 
&+ \Big[ \frac{\rho}{6}(\xi-x)^3 + \frac{\rho^2}{4}(\xi-x)^2 + \frac{\rho^3}{6}(\xi-x)+\frac{\rho^4}{24}\Big]\nabla^4u \nonumber \\ 
&+ \Big[ \frac{\rho}{24}(\xi-x)^4 + \frac{\rho^2}{12}(\xi-x)^3 + \frac{\rho^3}{12}(\xi-x)^2 + \frac{\rho^4}{24}(\xi-x)+\frac{\rho^5}{120}\Big]\nabla^5u \nonumber \\  
&+ \Big[  \frac{\rho}{120}(\xi-x)^5 + \frac{\rho^2}{48}(\xi-x)^4 + \frac{\rho^3}{36}(\xi-x)^3 + \frac{\rho^4}{48}(\xi-x)^2 + \frac{\rho^5}{120}(\xi-x)+\frac{\rho^6}{720}\Big]\nabla^6u \nonumber \\
&+ O(\tau_7).
\end{align}
The fact that $D_{\rho}u(\xi) - \nabla_{\rho}u = O(\tau_2)$ allows us to expand $\phi'_{\rho}(D_{\rho}u(\xi))$:
\begin{align}\label{45}
\phi'_{\rho}(D_{\rho}u(\xi)) =& \phi'_{\rho}(\nabla_{\rho}u) + \phi''_{\rho}(\nabla_{\rho}u)\big( D_{\rho}u(\xi) - \nabla_{\rho}u \big) + \1/2\phi'''_{\rho}(\nabla_{\rho}u)\big( D_{\rho}u(\xi) - \nabla_{\rho}u \big)^2 \nonumber \\ 
&+ \frac{1}{6}\phi^{(4)}_{\rho}(\nabla_{\rho}u)\big( D_{\rho}u(\xi) - \nabla_{\rho}u \big)^3 + \frac{1}{24}\phi^{(5)}_{\rho}(\nabla_{\rho}u)\big( D_{\rho}u(\xi) - \nabla_{\rho}u \big)^4 \nonumber \\
& + \frac{1}{120}\phi^{(6)}_{\rho}(\nabla_{\rho}u)\big( D_{\rho}u(\xi) - \nabla_{\rho}u \big)^5 + \frac{1}{720}\phi^{(7)}_{\rho}(\mu_7)\big( D_{\rho}u(\xi) - \nabla_{\rho}u \big)^6,
\end{align}
where $\mu_7 \in \conv\{\nabla_{\rho}u, D_{\rho}u(\xi)\}$.

%\begin{equation}
%\label{eq: zeta preserving 5th order polynomials}
%\sum_{\xi\in\Lambda}(a+b\xi+c\xi^2+d\xi^3+e\xi^4+f\xi^5)\zeta(x-\xi)=a+bx+cx^2+dx^3+ex^4+fx^5, \qquad \forall a,b,c,d\in\R.
%\end{equation} 
%\begin{lemma}
%\label{lemma15}
%Let $x,\rho \in \Omega$, $k=0,1,2,3,4,5$, then
%\begin{equation}
%\sum_{\xi \in \Lambda}\chi_{\xi,\rho}(x)(\xi-x)^k=\frac{(-\rho)^k}{k+1}.
%\end{equation}
%\end{lemma}
Combining \eqref{eq:StressA5}, \eqref{25}, \eqref{35}, \eqref{45}, after some algebraic manipulations,  we can obtain the stress error $R(u;x)$ in \eqref{Tylorexp5}. With a slight abuse of notation, for the convenience and clearness of the expression, we divide the $R(u;x)$ into several parts which indicates the different orders if we rescale the displacement $u$ by $\eps$. Using the identities \eqref{i5}, all the terms in $R(u;x)$ whose order are lower than six cancel out.

\begin{align}\label{Tylorexp5}
 &R(u;x) = \sum_{\xi\in\Lambda}\sum_{\rho \in \Rc} \Big\{ \rho\phi''_{\rho}(\nabla_{\rho} u)\Big[\rho(\xi-x)+\frac{\rho^2}{2}\Big]\nabla^2u \nonumber \\ 
 &-----------------------------------O(\eps) \nonumber \\ 
&+ \rho \phi''_{\rho}(\nabla_{\rho} u)\Big[\frac{\rho}{2}(\xi-x)^2+\frac{\rho^2}{2}(\xi-x)+\frac{\rho^3}{12}\Big]\nabla^3u\nonumber \\ 
%%%%%%%%%%%%%%%%
&+ \rho \phi'''_{\rho}(\nabla_{\rho} u)\Big[ \frac{\rho^2}{2}(\xi-x)^2 + \frac{\rho^3}{2}(\xi-x) + \frac{\rho^4}{6}\Big](\nabla^2 u)^2 \nonumber \\ 
&-----------------------------------O(\eps^2) \nonumber \\ 
%%%%%%%%%%%%%%%%%%%%%%%%%%%%
&+ \rho \phi''_{\rho}(\nabla_{\rho} u)\Big[ \frac{\rho}{6}(\xi-x)^3+\frac{\rho^2}{4}(\xi-x)^2+\frac{\rho^3}{6}(\xi-x)+\frac{\rho^4}{24}\Big]\nabla^4u \nonumber \\ 
%%%%%%%%%%%%%%%%%%%%%%%%%%%%
&+ \rho \phi'''_{\rho}(\nabla_{\rho} u)\Big[ \frac{\rho^2}{2}(\xi-x)^3+\frac{3\rho^3}{4}(\xi-x)^2+\frac{5\rho^4}{12}(\xi-x)+\frac{\rho^5}{12}\Big]\nabla^2u\nabla^3u \nonumber \\ 
%%%%%%%%%%%%%%%%%%%%%%%%%%%%
&+ \rho \phi^{(4)}_{\rho}(\nabla_{\rho} u)\Big[ \frac{\rho^3}{6}(\xi-x)^3 + \frac{\rho^4}{4}(\xi-x)^2 + \frac{\rho^5}{8}(\xi-x)+\frac{\rho^6}{48}\Big](\nabla^2u)^3 \nonumber \\ \
&-----------------------------------O(\eps^3) \nonumber \\
%%%%%%%%%%%%%%%%%%%%%%%%%%%%
&+\rho \phi''_{\rho}(\nabla_{\rho} u)\Big[ \frac{\rho}{24}(\xi-x)^4 + \frac{\rho^2}{12}(\xi-x)^3 +\frac{\rho^3}{12}(\xi-x)^2 +\frac{\rho^4}{24}(\xi-x)+ \frac{\rho^5}{180}\Big]\nabla^5u\nonumber \\
%%%%%%%%%%%%%%%%%%%%%%%%%%%%%
&+\rho \phi'''_{\rho}(\nabla_{\rho} u)\Big[ \rho(\xi-x)+\frac{\rho^2}{2}\Big]\Big[ \frac{\rho}{6}(\xi-x)^3+\frac{\rho^2}{4}(\xi-x)^2+\frac{\rho^3}{6}(\xi-x)+\frac{\rho^4}{24}\Big]\nabla^2u\nabla^4u \nonumber \\ 
%%%%%%%%%%%%%%%%%%%%%%%%%%%%%
&+\frac{\rho}{2}\phi'''_{\rho}(\nabla_{\rho} u)\Big[ \frac{\rho}{2}(\xi-x)^2+\frac{\rho^2}{2}(\xi-x)+\frac{\rho^3}{6}\Big]^2(\nabla^3u)^2 \nonumber \\ 
%%%%%%%%%%%%%%%%%%%%%%%%%%%%%
&+\frac{\rho}{6}\phi^{(4)}_{\rho}(\nabla_{\rho} u)\Big[ \rho(\xi-x)+\frac{\rho^2}{2}\Big]^2\Big[ \frac{\rho}{2}(\xi-x)^2+\frac{\rho^2}{2}(\xi-x)+\frac{\rho^3}{6}\Big](\nabla^2u)^2\nabla^3u \nonumber \\ 
%%%%%%%%%%%%%%%%%%%%%%%%%%%%%
&+\frac{\rho}{24}\phi^{(5)}_{\rho}(\nabla_{\rho} u)\Big[ \rho(\xi-x)+\frac{\rho^2}{2}\Big]^4(\nabla^2u)^4\nonumber \\ 
&-----------------------------------O(\eps^4)\nonumber \\ 
%%%%%%%%%%%%%%%%%%%%%%%%%%%%
&+\rho \phi''_{\rho}(\nabla_{\rho} u)\Big[ \frac{\rho}{120}(\xi-x)^5 + \frac{\rho^2}{48}(\xi-x)^4 + \frac{\rho^3}{36}(\xi-x)^3 +\frac{\rho^4}{48}(\xi-x)^2 +\frac{\rho^5}{120}(\xi-x)+ \frac{\rho^6}{720}\Big]\nabla^6u \nonumber \\ 
%%%%%%%%%%%%%%%%%%%%%%%%%%%%%
&+\rho \phi'''_{\rho}(\nabla_{\rho} u)\Big[ \frac{\rho}{2}(\xi-x)^2+\frac{\rho^2}{2}(\xi-x)+\frac{\rho^3}{6}\Big]\Big[ \frac{\rho}{6}(\xi-x)^3+\frac{\rho^2}{4}(\xi-x)^2+\frac{\rho^3}{6}(\xi-x)+\frac{\rho^4}{24}\Big]\nabla^3u\nabla^4u \nonumber \\ 
%%%%%%%%%%%%%%%%%%%%%%%%%%%%%
&+\rho \phi'''_{\rho}(\nabla_{\rho} u)\Big[ \rho(\xi-x)+\frac{\rho^2}{2}\Big]\Big[ \frac{\rho}{24}(\xi-x)^4+\frac{\rho^2}{12}(\xi-x)^3+\frac{\rho^3}{12}(\xi-x)^2+\frac{\rho^4}{24}(\xi-x)+\frac{\rho^5}{120}\Big]\nabla^2u\nabla^5u \nonumber \\ 
%%%%%%%%%%%%%%%%%%%%%%%%%%%%%
&+\frac{\rho}{2}\phi^{(4)}_{\rho}(\nabla_{\rho} u)\Big[ \rho(\xi-x)+\frac{\rho^2}{2}\Big]^2\Big[ \frac{\rho}{6}(\xi-x)^3+\frac{\rho^2}{4}(\xi-x)^2+\frac{\rho^3}{6}(\xi-x)+\frac{\rho^4}{24}\Big](\nabla^2u)^2\nabla^4u \nonumber \\ 
%%%%%%%%%%%%%%%%%%%%%%%%%%%%%
&+\frac{\rho}{2}\phi^{(4)}_{\rho}(\nabla_{\rho} u)\Big[ \rho(\xi-x)+\frac{\rho^2}{2}\Big]\Big[ \frac{\rho}{2}(\xi-x)^2+\frac{\rho^2}{2}(\xi-x)+\frac{\rho^3}{6}\Big]^2\nabla^2u(\nabla^3u)^2 \nonumber \\ 
%%%%%%%%%%%%%%%%%%%%%%%%%%%%%
&+\frac{\rho}{6}\phi^{(5)}_{\rho}(\nabla_{\rho} u)\Big[ \rho(\xi-x)+\frac{\rho^2}{2}\Big]^3\Big[ \frac{\rho}{2}(\xi-x)^2+\frac{\rho^2}{2}(\xi-x)+\frac{\rho^3}{6}\Big](\nabla^2u)^3\nabla^3u \nonumber \\ 
%%%%%%%%%%%%%%%%%%%%%%%%%%%%%
&+\frac{\rho}{120}\phi^{(6)}_{\rho}(\nabla_{\rho} u)\Big[ \rho(\xi-x)+\frac{\rho^2}{2}\Big]^5(\nabla^2u)^5 \nonumber \\ 
&-----------------------------------O(\eps^5) \nonumber \\ 
&+ O(\tau_7) + O(\tau_2\tau_6) + O(\tau_3\tau_5) + O(\tau^2_4) + O(\tau_5\tau^2_2) + O(\tau_2\tau_3\tau_4) + O(\tau^3_3) + O(\tau_4\tau^3_2) \nonumber \\
&+ O(\tau_3\tau^4_2) + O(\tau^2_3\tau^2_2) \Big\} \chi_{\xi, \rho}(x).
\end{align}

Hence we obtain the pointwise sixth order consistency estimate of the stress error $R(u;x)$. Following the analysis in Section \eqref{Sec::model}, Section \eqref{sec:stability} and Section \eqref{sec: a priori error estimate}, we can similarly prove the higher order continuum model \eqref{hocmodel6} is of sixth order accuracy.

\bibliographystyle{plain}
\bibliography{hoc.bib}

\begin{thebibliography}{10}

\bibitem{AM_A_Post_ML_FK_Model_2007}
M.~Arndt and M.~Luskin.
\newblock Goal-oriented atomistic-continuum adaptivity for the quasicontinuum
  approximation.
\newblock {\em Int. J. Multiscale Comput. Engrg.}, 5(49-50):407--415, 2007.

\bibitem{AM_A_Post_ML_FK_Model_2008_Model}
M.~Arndt and M.~Luskin.
\newblock Error estimation and atomistic-continuum adaptivity for the
  quasicontinuum approximation of a {F}renkel-{K}ontorova model.
\newblock {\em Multiscale Model. Simul.}, 7(1):147--170, 2008.

\bibitem{arndt2005derivation}
Marcel Arndt and Michael Griebel.
\newblock Derivation of higher order gradient continuum models from atomistic
  models for crystalline solids.
\newblock {\em Multiscale Model. Simul.}, 4(2):531--562, 2005.

\bibitem{BLBL:arma2002}
X.~Blanc, C.~Le~Bris, and P.-L. Lions.
\newblock From molecular models to continuum mechanics.
\newblock {\em Arch. Ration. Mech. Anal.}, 164(4):341--381, 2002.

\bibitem{MB_KH_Crystal_Latices}
M.~Born and K~Huang.
\newblock {\em Dynamical Theory of Crystal Lattices}.
\newblock Oxford Classic Texts in the Physical Sciences. Clarendon Press, 1954.

\bibitem{Braess2007}
D.~Braess.
\newblock {\em Finite Elements: Theory, Fast Solvers, and Applications in Solid
  Mechanics}.
\newblock Cambridge University Press; 3 edition, 2007.

\bibitem{Braun_2016_BC_for_AC}
J~Braun and B~Schmidt.
\newblock Existence and convergence of solutions of the boundary value problem
  in atomistic and continuum nonlinear elasticity theory.
\newblock {\em Calc. Var.}, 55(125):125--155, 2016.

\bibitem{cdkm06}
S.~Conti, G.~Dolzmann, B.~Kirchheim, and S.~M\"uller.
\newblock Sufficient conditions for the validity of the {C}auchy--{B}orn rule
  close to {SO}(n).
\newblock {\em J. Eur. Math. Soc.}, 8:515--530, 2006.

\bibitem{E:2007a}
W.~E and P.~Ming.
\newblock {C}auchy-{B}orn rule and the stability of crystalline solids: static
  problems.
\newblock {\em Arch. Ration. Mech. Anal.}, 183(2):241--297, 2007.

\bibitem{Guo2006Mechanical}
X.~Guo, J.~B. Wang, and H.~W. Zhang.
\newblock Mechanical properties of single-walled carbon nanotubes based on
  higher order cauchy?born rule.
\newblock {\em Int. J. Solids. Struct.}, 43(5):1276--1290, 2006.

\bibitem{hollig2003finite}
Klaus Hollig.
\newblock {\em Finite element methods with B-splines}, volume~26.
\newblock SIAM, 2003.

\bibitem{Hudson:stab}
T.~Hudson and C.~Ortner.
\newblock On the stability of {B}ravais lattices and their {C}auchy--{B}orn
  approximations.
\newblock {\em ESAIM Math. Model. Numer. Anal.}, 46:81--110, 2012.

\bibitem{bqce}
X.~H. Li, C.~Ortner, A.~Shapeev, and B.~Van Koten.
\newblock Analysis of blended atomistic/continuum hybrid methods.
\newblock {\em Numer. Math.}, 134(2):275--326, 2016.

\bibitem{LuOr:acta}
M.~Luskin and C.~Ortner.
\newblock Atomistic-to-continuum-coupling.
\newblock {\em Acta Numerica}, 22:397--508, 2013.

\bibitem{ML_CO_BK_Energy_Blended_QC_2011}
M.~Luskin, C.~Ortner, and B.~Van~Koten.
\newblock Formulation and optimization of the energy-based blended
  quasicontinuum method.
\newblock {\em Comput. Methods Appl. Mech. Engrg.}, 253:160--168, 2013.

\bibitem{Lyu2017Multiscale}
Dandan Lyu and Shaofan Li.
\newblock Multiscale crystal defect dynamics: A coarse-grained lattice defect
  model based on crystal microstructure.
\newblock {\em J. Mech. Phys. Solids}, 107:379--410, 2017.

\bibitem{Ortner:qnl.1d}
C.~Ortner.
\newblock A priori and a posteriori analysis of the quasi-nonlocal
  quasicontinuum method in {1D}.
\newblock {\em Math. Comp.}, 80(275):1265--1285, 2011.

\bibitem{AlexIntp:2012}
C.~Ortner and A.~V. Shapeev.
\newblock Interpolation of lattice functions and applications to
  atomistic/continuum multiscale methods.
\newblock {\em ArXiv e-prints}, arXiv:1204.3705, 2012.

\bibitem{OrtnerSuli:2008a}
C.~Ortner and E.~S{\"u}li.
\newblock Analysis of a quasicontinuum method in one dimension.
\newblock {\em ESAIM Math. Model. Numer. Anal.}, 42(1):57--91, 2008.

\bibitem{OrtnerTheil2012}
C.~Ortner and F.~Theil.
\newblock Justification of the cauchy--born approximation of elastodynamics.
\newblock {\em Arch. Ration. Mech. Anal.}, 207(3):1025--1073, 2013.

\bibitem{OrtnerWang:2011}
C.~Ortner and H.~Wang.
\newblock A priori error estimates for energy-based quasicontinuum
  approximations of a periodic chain.
\newblock {\em Math. Models Methods Appl. Sc.}, 21:2491--2521, 2011.

\bibitem{CO_HW_ACC_1D_A_Posteriori}
C.~Ortner and H.~Wang.
\newblock A posteriori error control for a quasi-continuum approximation of a
  periodic chain.
\newblock {\em IMA J. Numer. Anal.}, 34(3):977--1001, 2014.

\bibitem{2012grac}
Christoph Ortner and Lei Zhang.
\newblock Construction and sharp consistency estimates for atomistic/continuum
  coupling methods with general interfaces: A two-dimensional model problem.
\newblock {\em SIAM J. Numer. Anal.}, 50(6):2940--2965, 2012.

\bibitem{2016bgfc}
Christoph Ortner and Lei Zhang.
\newblock Atomistic/continuum blending with ghost force correction.
\newblock {\em SIAM J. Sci. Comput.}, 38(1), 2016.

\bibitem{peridynamics2009}
Pablo Seleson, Michael Parks, Max Gunzburger, and Richard B.~Lehoucq.
\newblock Peridynamics as an upscaling of molecular dynamics.
\newblock {\em Multiscale Model. Simul.}, 8:204--227, 01 2009.

\bibitem{sun2008application}
Yuzhou Sun and Kim~Meow Liew.
\newblock Application of the higher-order cauchy--born rule in mesh-free
  continuum and multiscale simulation of carbon nanotubes.
\newblock {\em Internat. J. Numer. Methods Engrg.}, 75(10):1238--1258, 2008.

\bibitem{Sunyk2003On}
R.~Sunyk and P.~Steinmann.
\newblock On higher gradients in continuum-atomistic modelling.
\newblock {\em Int. J. Solids. Struct.}, 40(24):6877--6896, 2003.

\bibitem{Triantafyllidis1993}
N.~Triantafyllidis and S.~Bardenhagen.
\newblock On higher order gradient continuum theories in 1-d nonlinear
  elasticity. derivation from and comparison to the corresponding discrete
  models.
\newblock {\em J. Elasticity}, 33(3):259--293, Dec 1993.

\bibitem{BK_ML_Energy_Blended_QC_2011}
B.~Van~Koten and M.~Luskin.
\newblock Analysis of energy-based blended quasicontinuum approximations.
\newblock {\em SIAM J. Numer. Anal.}, 49(5):2182--2209, 2011.

\bibitem{HW_ML_PL_LZ_Adaptiv_AC_2D}
H.~Wang, M.~Liao, P.~Lin, and L.~Zhang.
\newblock A posteriori error estimation and adaptive algorithm for the
  atomistic/continuum coupling in two dimensions.
\newblock {\em SIAM J. Sci. Comput.}, 40(4):A2087--A2119, 2018.

\bibitem{HW_SL_FY_A_Post_BQCF_2018}
H.~Wang, S.~Liu, and Yang. F.
\newblock A posteriori error control for three typical force-based
  atomistic-to-continuum coupling methods for an atomistic chain.
\newblock {\em Numer. Math. Theor. Meth. Appl.}, 12:233--264, 2018.

\bibitem{HW_SY_2018_Efficiency_A_Post_1D}
H.~Wang and S.~Yang.
\newblock Analysis of the residual type and the recovery type a psoteriori
  error estimators for a consistent atomistic-to-continuum coupling method in
  one-dimension.
\newblock {\em Multiscale Model. Simul.}, 16(3):679--709, 2018.

\end{thebibliography}

\end{document}